\documentclass{amsart}
\usepackage{hyperref}
\usepackage{times}
\usepackage{verbatim}
\usepackage[T1]{fontenc}
\usepackage[utf8]{inputenc}
\usepackage[english]{babel}
\usepackage{amssymb}
\usepackage{amsthm}
\usepackage{amsmath}
\usepackage{amstext}
\usepackage{multirow}
\usepackage{graphicx}
\usepackage[all]{xy}

\newtheorem{THM}{Theorem}
\newtheorem{thm}{Theorem}[section]

\newtheorem{lemma}[thm]{Lemma}
\newtheorem{prop}[thm]{Proposition}

\newtheorem{cor}[thm]{Corollary}

\newtheorem{example}[thm]{Example}
\newtheorem{remark}[thm]{Remark}

\newcommand{\mb}{\mathbb}
\newcommand{\mc}{\mathcal}
\newcommand{\R}{\mb R}
\newcommand{\C}{\mb C}
\newcommand{\PP}{\mb P}
\newcommand{\Z}{\mb Z}
\newcommand{\Q}{\mb Q}
\newcommand{\N}{\mb N}

\newcommand{\OO}{\mc O}
\newcommand{\F}{\mc F}
\newcommand{\G}{\mc G}

\DeclareMathOperator{\utype}{\mbox{utype}}
\DeclareMathOperator{\Tang}{\mbox{Tang}}
\DeclareMathOperator{\ord}{\mbox{ord}}
\DeclareMathOperator{\Hom}{\mbox{Hom}}

\newcommand{\Fol}{\mbox{Fol}(U,C)}
\newcommand{\Folf}{\widehat{\mbox{Fol}}(U,C)}
\newcommand{\Diff}{\mathrm{Diff}(\C,0)}
\newcommand{\Difffor}{\widehat{\mathrm{Diff}}(\C,0)}

\newcommand{\closed}{\mc C}
\newcommand{\lattice}{\Gamma_\tau}

\usepackage{mathtools}
\usepackage{ifmtarg}
\addtocounter{section}{0}             % Start with section 1
\numberwithin{equation}{section}       % Number formulas within sections

\title[Neighborhoods of elliptic curves]{Two dimensional neighborhoods of elliptic curves: formal classification and foliations.}

\author[F. Loray, O. Thom and F. Touzet]{Frank Loray, Olivier Thom and Fr\'ed\'eric Touzet}
\address{Univ Rennes, CNRS, IRMAR - UMR 6625, F-35000 Rennes, France}
\email{frank.loray@univ-rennes1.fr}
\email{frederic.touzet@univ-rennes1.fr}
\email{olivier.thom@univ-rennes1.fr}

\thanks{Supported by ANR-16-CE40-0008 project ``Foliage''}
\thanks{MSC 32}
\thanks{ Key words: Elliptic curves, formal neighborhoods}

\begin{document}
\begin{abstract} 
We classify two dimensional neighborhoods of an elliptic curve $C$ with torsion normal bundle, up to formal equivalence.  The proof makes use of the existence of a pair (indeed a pencil) of formal foliations having  $C$ as a common leaf, and the fact that neighborhoods are completely determined by the holonomy of such a pair. We also discuss
analytic equivalence and for each formal model, we show that the corresponding moduli space is infinite dimensional.
\end{abstract}
\maketitle

\sloppy
\tableofcontents

\section{Introduction}

Let $C$ be a smooth elliptic curve: $C=\C/\lattice$, where $\lattice=\Z+\tau\Z$.
Given an embedding $\iota:C\hookrightarrow U$ of $C$ into a smooth complex surface $U$,
we would like to understand the germ $(U,\iota(C))$ of neighborhood of $\iota(C)$ in $U$.
Precisely, we will say that two embeddings $\iota,{\iota}':C\hookrightarrow U,U'$ are (formally/analytically) equivalent
if there is a (formal/analytic) isomorphism $\phi:(U,\iota(C))\to(U',\iota'(C))$ between germs of neighborhoods
making commutative the following diagram
\begin{equation}\label{eq:equivNeighbId}
\xymatrix{\relax
    C \ar[r]^\iota \ar[d]_{\text{id}}  & U \ar[d]^\phi \\
    C \ar[r]^{\iota'} & U'
}\end{equation}
We want to understand the (formal/analytic) classification of such neighborhoods, i.e. up to above equivalence.
By abuse of notation, we will still denote by $C$ the image $\iota(C)$ of its embedding in $U$.

\subsection{Some known results}
A first invariant is the normal bundle $N_C$ of $C$ in the surface $U$, or we better should say, its class in 
the Picard group $\mathrm{Pic}(C)$. Its degree $d=\deg(N_C)$ coincide with the self-intersection $C\cdot C$
of the curve in the surface: it is the unique topological invariant. 
We say that the germ of neighborhood $(U,C)$
is {\bf linearizable} if it is equivalent to the germ of neighborhood of the zero section in the total space of $N_C$:
$(U,C)\sim(N_C,0)$. If $(U,C)$ is such that $C\cdot C<0$, then it is analytically linearizable following Grauert \cite{Grauert} 
(see also \cite{Laufer,CamachoMovasati}).

Ilyashenko studied the case $C\cdot C>0$ in \cite{Ilyashenko} and, for any fixed normal bundle $N_C$ with $d=\deg(N_C)>0$,
he constructed a huge holomorphic family of non formally equivalent germs $(U,C)$. 
The family is parametrized by the set of germs of holomorphic maps $(\C^2,0)\to(\C^d,0)$
and is in some sense the universal deformation of the linear model $(N_C,0)$ (see \cite{Ilyashenko} for more details).

{\bf Non torsion case and diophantine condition.}
So far, and like for neighborhoods of rational curves (see \cite{Grauert,Savelev,Mishustin0,MaycolFrank}), 
formal and analytic classifications coincide.
This is no longer true when we consider zero type neighborhoods, i.e. with $C\cdot C=0$, as shown by Arnold
in \cite{Arnold}. 
More precisely, Arnold proved that the neighborhood $(U,C)$ is formally linearizable when $N_C$ is {\bf not torsion}, 
and that it is also analytically linearizable provided that $N_C$ satisfies some diophantine condition:
\begin{equation}\label{eq:diophantine}
\exists\ \epsilon,\alpha>0\ \text{such that}\ \forall\ k\in\N,\ \ \  d(N_C^{\otimes k},\OO_C)\ge\frac{\epsilon}{k^\alpha}
\end{equation}
where $d$ is the distance given by the natural flat metric on the jacobian variety $\mathrm{Jac}(C)\simeq C$.
This means that $N_C$ has bad approximation by torsion bundles. 
This diophantine condition is satisfied for a set of full Lebesgue measure in $\mathrm{Jac}(C)$.
However, for somme "exceptional" bundles $N_C$ two well approximated by torsion ones, 
there are examples of non analytically linearizable neighborhoods: they contain infinitely many elliptic curves 
accumulating on $C$. In fact, there is a deep connection between the linearization of neighborhoods $(U,C)$ 
and that of germs of one variable diffeomorphisms at an indifferent fixed point:
$$\varphi(z)=az+\sum_{n>1}a_nz^n,\ \ \ \vert a\vert=1.$$
Arnold's Theorem is a reminiscence of Siegel's Linearization Theorem, while non linearizable example 
are resurgences of Kremer's non linearizable dynamics. Since Arnold's work, deep progresses have been 
done about the one dimensional problem. Linearizability's diophantine condition has been weakened by Brjuno,
and its optimality has been proved by Yoccoz. For each multiplier $a$ violating Brjuno's condition, 
Yoccoz constructed non analytically linearizable dynamics $\varphi(z)=az+o(z)$ with infinite degree of freedom
giving rise again to an infinite dimensional moduli space; moreover, P\'erez-Marco provided examples 
without periodic orbits except the fixed point $z=0$, and with huge centralizer. 

Even if there has been less progress in the neighborhood side of the story, it is interesting to note
that one can construct non linearizable neighborhoods $(U,C)$ by suspension of non linearizable dynamics $\varphi$.
This is how foliations come into the story. The linear bundle $N_C$ is flat, admitting a $1$-parameter family
of (flat) holomorphic connections. On the total space of $N_C$, they define a pencil of foliations $(\F_t)_{t\in\PP^1}$ 
with the following properties:
\begin{itemize}
\item $\F_t=\{\omega_t=0\}$ with $\omega_t=\omega_0+t\omega_\infty$,
\item $\omega_0$ is a closed logarithmic $1$-form with pole along $C$ having residue $1$ 
and purely imaginary periods along $C$ (i.e. the associated foliation $\F_0$ has unitary holonomy);
\item $\omega_\infty$ is holomorphic, transversal to $C$, inducing the holomorphic $1$-form $dx$ on $C$.
\end{itemize}
Therefore, each foliation $\F_t$ is regular; for $t\in\C$, the curve $C$ is a leaf of $\F_t$, and its holonomy is
\begin{equation}\label{eq:holPencilLinear}
\pi_1(C)\to\mathrm{GL}_1(\C)\simeq\C^*\ ;\ 
\left\{\begin{matrix}
1\mapsto e^{\alpha_1+t}\hfill\\ \tau\mapsto e^{\alpha_\tau+t\tau}
\end{matrix}\right.
\end{equation}
where $\alpha_1,\alpha_\tau$ are the fundamental periods of $\omega_0$. In particular, $\F_0$ is the unique 
unitary foliation. Going back to $(U,C)$, we get a pencil of formal foliations $(\hat{\F}_t)$ and it can be 
shown (see section \ref{sec:FoliationNonTorsion}) that there is no other formal foliation tangent to $C$, 
i.e. having $C$ as a leaf. 

On the other hand, the works of Yoccoz and P\'erez-Marco provide many representations 
$$\pi_1(C)\to\Diff\ ;\ 
\left\{\begin{matrix}
1\mapsto a_1z+o(z)\hfill\\ \tau\mapsto a_\tau z+o(z)
\end{matrix}\right.$$
with $\vert a_1\vert=\vert a_\tau\vert=1$ which are formally but non analytically linearizable.
By suspension, we can construct a neighborhood $(U,C)$ equipped with $2$ foliations
$\F_0$ and $\F_\infty$, respectively tangent and transversal to $C$, and such that 
the holonomy of $\F_0$ is precisely given by the given representation. If $(U,C)$
were analytically linearizable, $\F_0$ would be the unitary linear foliation,
contradicting the non linearizability assumption.

{\bf Torsion case, fibration and Ueda type.} 
When $N_C$ is torsion of order say $m:=\ord(N_C)$, the unitary foliation $\F_0$, defined on the total space of $N_C$,
has torsion holonomy of order $m$. It thus defines a fibration $f:N_C\to\C$ for which $mC$ is a special fiber.
For the neighborhoods $(U,C)$, there is a strong dichotomy: either the fibration persists, or not.
The former case is dealt with the theory of Kodaira-Spencer. If there is a formal fibration tangent to $C$, 
then it is actually analytic; moreover, it is analytically isotrivial if and only if $U$ admits a formal fibration transversal to $C$.
In this situation, we can define $\kappa\in\N^*\cup\{\infty\}$ such that $\F_0$ (or $f$) is isotrivial up to order $\kappa$, 
but not up to order $\kappa+1$.
In other words, if $I\subset\OO_U$ is the ideal sheaf defining $C$, and $C(n):=\mathrm{Spec}(\OO_U/I^{n+1})$ 
denotes the $n^{\text{th}}$ infinitesimal neighborhhood, then $C(\kappa)$ is the largest one 
that admits a fibration transversal to $C$.
This is the only invariant, which means that any two germs $(U,C)$ and $(U',C)$ having the same normal bundle $N_C$
and integer $\kappa=\kappa'$ are analytically equivalent. 

When there is no formal fibration, Ueda introduced in \cite{Ueda} another invariant, namely 
the Ueda type of $(U,C)$ (so named by Neeman in \cite{Neeman}) which is, similarily to $\kappa$,
the first obstruction to define $\F_0$ (see section \ref{sec:Uedatype}): 
{\it $\utype(U,C)=k\in\N^*\cup\{\infty\}$ if $C(k)$ admits a fibration tangent to $C$, but not $C(k+1)$.}
The Ueda type is a multiple of the torsion $m$ of the normal bundle $N_C$. 
When $\utype(U,C)<\infty$, any formal holomorphic function on $(U,C)$ 
is constant; indeed, such a function must be constant along $C$ and 
would then define a formal fibration if it were non constant, contradicting the finiteness of the Ueda type.

\subsection{Foliated structure of the neighborhoods} 
The main goal of the paper is to provide a formal classification of germs of neighborhoods $(U,C)$
in the missing case: $N_C$ is torsion and $\utype(U,C)=k<\infty$. This is a reminiscence
of classification of diffeomorphisms with torsion linear part
$$\varphi(z)=az+o(z)\in\Diff,\ \ \ a^k=1.$$
To this aim, we prove (see the end of section \ref{S:BifClass})

\begin{THM}\label{THM:PencilFol}
When $N_C$ is torsion and $\utype(U,C)=k<\infty$, there exists a pencil of formal foliations 
$(\F_t)_{t\in\PP^1}$ on $(U,C)$ having the following properties:
\begin{itemize}
\item $\F_t=\{\omega_t=0\}$ where $\omega_t=\omega_0+t\omega_\infty$ are formal closed meromorphic $1$-forms 
with polar locus supported on $C$ and whose multiplicities are $k+1$ for $t\in\C$ and $p+1$, $-1\leq p<k$ for $\omega_\infty$.
\item there exists a transversal $(T,z)$ such for all $t\in\mathbb C$, the foliation $\F_t$ is tangent to $C$ and its holonomy
has the form
\begin{equation}\label{eq:HolonPencilUedak}
\pi_1(C)\to\Difffor\ ;\ 
\left\{\begin{matrix}
1\mapsto a_1[z+t z^{k+1}+o(z^{k+1})]\hfill\\ \tau\mapsto a_\tau[z+(1+t\tau) z^{k+1}+o(z^{k+1})]
\end{matrix}\right.
\end{equation}
where $(a_1,a_\tau)$ is the $m$-torsion monodromy of the unitary connection on $N_C$;
\item the holonomy of $\F_0$ along the cycle $1\in\lattice$ is torsion, and this characterizes $\F_0$ in the pencil;
\item $\F_\infty$ is either transversal to $C$, or is tangent to $C$ and its holonomy
takes the form
\begin{equation}\label{eq:HolonPencilUedakappa}
\pi_1(C)\to\Difffor\ ;\ 
\left\{\begin{matrix}
1\mapsto a_1e^cz+o(z)\hfill\\ \tau\mapsto a_\tau e^{c\tau}z+o(z)
\end{matrix}\right.
\ \text{or}\ 
\left\{\begin{matrix}
a_1[z+cz^{p+1}+o(z^{p+1})]\hfill\\ a_\tau[z+c\tau z^{p+1}+o(z^{p+1})]
\end{matrix}\right.
\end{equation}
with $c\in\C^*$ and $0<p<k$ with $p\in m\N^*$.
\end{itemize}
Moreover, all formal regular foliations tangent/transversal to $C$ on $U$ are  contained in the pencil.
\end{THM}

\subsection{Formal classification}
The formal classification of neighborhoods $(U,C)$ turns out to be equivalent
to the formal classification of bifoliated neighborhoods $(U,C,\F_0,\F_\infty)$.
When $\F_\infty$ is a fibration transversal to $C$, then it is a suspension,
and the classification reduces to that of the holonomy representation of $\F_0$:
there is one formal invariant $\lambda\in\C$ in this case. 
When $\F_\infty$ is tangent to $C$, we are led to classify pairs of foliations $(\F,\G)$ having $C$ as a common leaf.
A first invariant is given by the contact order $\Tang(\F,\G)\in\N^*$ between the two foliations along $C$ (see section \ref{sec:BifBasicInv}).
For instance, in Theorem \ref{THM:PencilFol}, we have $\Tang(\F_t,\F_{t'})=k+1$ while $\Tang(\F_t,\F_\infty)=p+1$.
Given a transversal $(T,y)$ to $C$, we can consider the holonomy representations 
$$\rho_{\F},\rho_{\G}:\pi_1(C)\to\Difffor\ (\text{or}\ \Diff).$$
Then we prove (see the end of section \ref{S:BifClass})

\begin{THM}\label{ThmBifoliated} $\ $

 \noindent {\bf Classification.} Any formal/analytic bifoliated neighborhoods $(U,C,\F,\G)$ and $(\tilde U,C,\tilde{\F},\tilde{\G})$ 
with same contact order 
$\Tang(\F,\G)=\Tang(\tilde{\F},\tilde{\G})$  are for\-mal\-ly/analytically equivalent if, and only if, there 
exist formal/analytic 
diffeomorphisms  $\phi,\psi\in\Difffor/\Diff$ \text{such that} for all $\gamma\in\pi_1(C)$
\begin{equation}\label{eq:Bifolholclas}
\left\{\begin{matrix}\rho_{\F}(\gamma)\circ \phi=\phi\circ\rho_{\tilde{\F}}(\gamma)\\ \rho_{\G}(\gamma)\circ \psi=\psi\circ\rho_{\tilde{\G}}(\gamma)
\end{matrix}\right.
\mbox{and such that}\ \phi=\psi\ \text{mod}\ y^{k+2}.
\end{equation}
{\bf Realization.} 
Given two formal/analytic representations $\rho_1,\rho_2$ 
of the form  (\ref{eq:HolonPencilUedak}) or (\ref{eq:HolonPencilUedakappa})
for distinct $t_1,t_2\in\PP^1$, there is a unique formal/analytic bifoliated 
neighborhood $(U,C,\F_{t_1},\F_{t_2})$ having holonomy representations $\rho_{1},\rho_{2}$ on a given transversal.
\end{THM}

Applying Theorem \ref{ThmBifoliated} to the canonical pair $(\F_0,\F_\infty)$ of Theorem \ref{THM:PencilFol},
we get the formal classification (see Theorems \ref{thm:NormalFormNeigh0=p<k} and \ref{thm:NormalFormNeigh0<p<k}).

\begin{THM}\label{TH:FORMAL_CLASS}
The set of formal equivalent classes of germs $(U,C)$ with given normal bundle 
$N_C$ of order $m$ and Ueda type $k=mk'$ is in one-to-one correspondance with 
$$(\lambda,\underbrace{(\lambda_0,\lambda_1,\ldots,\lambda_{k'-1})}_{\Lambda})\in\C\times\C^{k'}$$
up to the action of $k'^{\text{th}}$ roots of unity defined by $(\lambda_i)\stackrel{\mu}{\mapsto} (\mu^{-i}\lambda_i)$.
\end{THM}

Moreover, for each formal class $(\lambda,\Lambda)$, a representative $(U_{\lambda,\Lambda},C)$ is given by the quotient 
of the germ of neighborhood $(\tilde U,\tilde C)$ of $\tilde C:=\{y=0\}$ in $\tilde U:=\C_x\times\C_y$ by the action 
of $\lattice$ generated by:
\begin{equation}\label{eq:NormalFormNeighGroupGeral}
\left\{\begin{array}{ccc} \phi_1(x,y)&=& (x+1\ ,\ a_1 y) \\ 
\phi_\tau(x,y)&=&(x+\tau+g_{k,\lambda,\Lambda}(y)\ ,\ a_\tau \varphi_{k,\lambda}(y))\end{array}\right. 
\end{equation}
$$\text{where}\ \ \ \varphi_{k,\lambda}=\exp\left(\frac{y^{k+1}}{1+\lambda y^k}\partial_y\right),\ \ \ g_{\lambda,\Lambda}=\int_0^y a_\tau\varphi_{k,\lambda}^*\omega_\Lambda-\omega_\Lambda$$
$$\text{with}\ \ \ \omega_\Lambda:=P(\frac{1}{y^m})\frac{dy}{y},\ \ \ P(z):=\sum_{i=0}^{k'-1}\lambda_i z^i.$$
The pencil $\omega_t=\omega_0+t\omega_\infty$ of closed $1$-forms of Theorem \ref{THM:PencilFol} is generated by
\begin{equation}
\omega_0=\frac{dy}{y^{k+1}}+\lambda\frac{dy}{y}\ \ \ \text{and}\ \ \ \omega_\infty=dx-\omega_\Lambda.
\end{equation}
The holonomy representation of $\F_0$ computed on the transversal $\{x=0\}$ is given by
\begin{equation}\label{eq:NormalFormNeighHolF0}
\rho_{\F_0}\ :\ 
\left\{\begin{matrix}
1\mapsto a_1 y\hfill\\ \tau\mapsto a_\tau \varphi_{k,\lambda}(y)
\end{matrix}\right.
\end{equation}
When $\Lambda=0$, the germ $(U_{\lambda,0},C)$ is just the suspension of this representation 
and $\F_\infty$ is the fibration transversal to $C$. When $\Lambda\not=0$, $C$ is a common leaf of $\F_\infty$ and $\F_0$ and the tangency order of these two foliations along $C$ is $p=m\deg(P)$.  Moreover, the holonomy representation of $\F_\infty$ is given by
\begin{equation}\label{eq:NormalFormNeighHolFinfty}
\rho_{\F_\infty}\ :\ 
\left\{\begin{matrix}
1\mapsto a_1\exp(v_\Lambda)\hfill\\ \tau\mapsto a_\tau \exp(\tau v_\Lambda)
\end{matrix}\right.\ \ \ \text{where}\ v_\Lambda=\frac{y}{P(\frac{1}{y^m})}\partial_y.
\end{equation}

In the appendix,
we also consider the action of automorphisms of the elliptic curve $C$ on the normal forms $(U_{\lambda,\Lambda})$.

\subsection{About analytic classification} 
For a general neighborhood $(U,C)$, with $N_C$ torsion and $\utype(U,C)=k<\infty$,
foliations $\F_t$ of Theorem \ref{THM:PencilFol} are only formal. Moreover, Mishustin recently gave in \cite{Mishustin}
an example where all $\F_t$ are divergent. 
Using the first part of Theorem \ref{ThmBifoliated}, we can prove (section \ref{S:criteriaconv})

\begin{THM}\label{thm:3convergent} Let $(U,C)$ be an analytic neighborhood with $N_C$ torsion and $\utype(U,C)=k<\infty$.
Assume
\begin{itemize}
\item three elements $\F_{t_1},\F_{t_2},\F_{t_3}$ of the pencil are convergent,
\item or two elements $\F_{t_1},\F_{t_2}$ of the pencil are convergent with $\frac{t_i}{1+t_i\tau}\not\in\Q$ for $i=1,2$;
\end{itemize}
then the full pencil $(\F_t)$ is convergent and $(U,C)$ is analytically equivalent to its formal normal form.
\end{THM}

On the other hand, using the second part of Theorem \ref{ThmBifoliated}, we can prove (section \ref{S:criteriaconv})

\begin{THM}\label{thm:EmbedEcalleVoronin} 
For each formal normal form $(U_0,C)$ and each pair 
$t_1,t_2\in\PP^1$ with $\frac{t_1}{1+t_1\tau}\in\Q$, then the set of analytic equivalence classes of analytic
neighborhoods $(U,C)$ formally equivalent to $(U_0,C)$ with $\F_{t_1},\F_{t_2}$ convergent 
is infinite dimensional.
\end{THM}

In fact, we prove that each of these deformation spaces contains \'Ecalle-Voronin moduli space.
In the non torsion case, we can also use a construction similar to the second part of Theorem \ref{ThmBifoliated}
to construct many examples with only two convergent foliations $\F_0,\F_t$, $t\in\C^*$, but with divergent transversal fibration. In fact, we realize Yoccoz non linearizable dynamics (see \cite{PerezMarco}) as holonomy
for $\F_0$. Again, there are infinitely many parameters. So far, all non linearizable examples
of neighborhoods, in the non torsion case, were obtained by suspension, i.e. with an analytic 
fibration transversal to $C$.

\subsection{Concluding remarks when $U$ is projective}
We do not know any projective exemple with $N_C$ non torsion and $(U,C)$ non linearizable.
In \cite{Paulo}, Sad shows that the blow-up $(U,C)\to(\PP^2,C_0)$ at $9$ points on a smooth cubic $C_0$ 
admits a global foliation having $C$ as a regular leaf only in the torsion case $N_C$, which is a
fibration in this case. Therefore, the local analytic foliation that exists in the very generic non torsion case
(when $N_C$ satisfies diophantine equation (\ref{eq:diophantine})) does not extend (even as a singular foliation)
on the whole of $U$. In the torsion case, Neeman proved in \cite[Theorem 6.12]{Neeman} that any smooth projective
$(U,C)$ with $\utype(U,C)<\infty$ is a blow-up (outside of the elliptic curve) of a unique model $(U_0,C)$
(corresponding to $m=k=1$ and $\lambda=\Lambda=0$ in Theorem \ref{TH:FORMAL_CLASS}), namely $U_0=\PP(E)$ where $E$ is the unique non trivial extension $0\to \OO_C\to E\to \OO_C\to 0$ and $C$ is regarded as a section via the embedding $\OO_C\to E$, with $N_C=\OO_C$.\\
Brunella proved in \cite[Chapter 9, Corollary 2]{Brunella} that any projective $(U,C)$ with a global
(possibly singular) foliation having $C$ as a compact leaf is, up to ramified coverings and birational maps,
either an elliptic fibration with $C$ as a (possibly multiple) fiber, or a ruled surface over $C$ admitting
 a section. In every cases, the germ $(U,C)$ is analytically linearizable, except if $(U,C)$ fits into 
 Neeman's above examples. In \cite{leaves}, the authors investigate similar global questions for neighborhoods of higher genus 
curves in surfaces, or higher dimensional hypersurfaces. In the local setting, at the neighborhood of a curve $C$
of arbitray genus with flat normal bundle in a complex surface, there still exists many formal foliations having $C$
as a regular leaf (see \cite[Theorem A]{leaves}). Also, a bifoliated classification can be done, at least in 
genus $2$ (see \cite{Olivier,TheseOlivier}).

\section{Foliations, holonomy and closed $1$-forms}\label{S:Abeliansubgroups}

The principal tool that we will use to classify germs of neighborhoods $(U,C)$ will be 
the existence of formal foliations that we will prove in the next section.
Let 
$$\Difffor=\{\sum_{n>0}a_nz^n\ ;\ a_1\not=0\}$$ 
denotes the group of formal diffeomorphisms fixing $0$
and $\Diff\subset\Difffor$ the subgroup of holomorphic germs (i.e. convergent ones).

\subsection{Foliations and holonomy}\label{sec:folhol}
A formal regular foliation $\F$ on $(U,C)$ tangent to $C$ (i.e. $C$ is a leaf)
is defined by a covering $(U_i)$ together with formal submersions $f_i:U_i\to\C$
with $f_i(C\cap U_i)=0$ satisfying on each non empty intersection $U_i\cap U_j$ 
\begin{equation}\label{eq:submdefolf}
f_i=\varphi_{ij}\circ f_j\ \ \ \text{for some}\ \varphi_{ij}\in\Difffor.
\end{equation}
We denote by $\Folf$ the set of such foliations, and by $\Fol$ those ones that can be defined 
with convergent $f_i$'s.

An element $\F\in\Folf$ has a holonomy representation 
\begin{equation}\label{eq:holonomyrepgeral}
\rho_{\F}\ :\ \pi_1(C)\to\Difffor.
\end{equation}
Precisely, given a loop $\gamma:[0,1]\to C$ based in $p\in C$, we can cover $\gamma$ by a finite sequence 
of these open sets, say $U_0,U_1,\ldots,U_n,U_0$ such that $\gamma$ intersect successively these open sets
in this order. Then analytic continuation of $f_0$ along $\gamma$ is obtained by
\begin{equation}\label{eq:defholfor}
f_0=\varphi_{01}\circ f_1=\varphi_{01}\circ\varphi_{12}\circ  f_2=\cdots=\underbrace{\varphi_{01}\circ\varphi_{12}\circ\cdots\circ\varphi_{n0}}_{\varphi^\gamma}\circ f_0
\end{equation}
and we define the morphism (\ref{eq:holonomyrepgeral}) by $[\gamma]\mapsto\varphi^\gamma$. 
Changing $f=f_0$ by another first integral, say $\phi\circ f$ with $\phi\in\Difffor$, 
will conjugate $\rho_{\F}$ by $\phi$:
$$(\phi\circ f)^\gamma=\phi\circ f^\gamma=\phi\circ\varphi^\gamma\circ f
=(\phi\circ\varphi^\gamma\circ\phi^{-1})\circ(\phi\circ f)$$
But the class of the representation $\rho_{\F}$ up to conjugacy in $\Difffor$ does not depend
on the choice of $f_i$'s, of $U_i$'s and even of the base point.
If $\F$ is holomorphic (i.e. convergent), then $\rho_{\F}$ takes values in the group
$\Diff=\{\sum_{n>0}a_nz^n\in\C\{z\}\ ;\ a_1\not=0\}$ of germs of holomorphic diffeomorphisms
fixing $0$. The fundamental group $\pi_1(C)\simeq \lattice$ is abelian and the classification of 
abelian subgroups of $\Difffor$ is well-known.

\subsection{Abelian subgroups of $\Difffor$}\label{SSS:normalforms}
We give here an enumeration of the main properties of such subgroups (at least those useful in the sequel) for which we refer to \cite{Frankpseudo} and references therein.
Let us start by recalling that any holomorphic vector field $v=f(z)\partial_z$ fixing $0$, $f(0)=0$,
is holomorphically conjugated to one and only one of the following vector fields:
\begin{equation}\label{eq:normformvectfield}
v_\alpha=\alpha z\partial_z\ \ \ \text{or}\ \ \ v_{k,\lambda}=\frac{z^{k+1}}{1+ {\lambda} {z^k}} \partial z,
\end{equation}
where $k\in\N^*$ and $\alpha,\lambda\in\C$. This comes form the fact that the dual meromorphic $1$-form defined by $\omega(v)\equiv 1$ is characterized up to change of coordinate by its order of pole, and its residue: it is therefore conjugated to only one model
\begin{equation}\label{eq:normformoneform}
\omega_\alpha=\alpha^{-1} \frac{dz}{z}\ \ \ \text{or}\ \ \ \omega_{k,\lambda}=\frac{dz}{z^{k+1}}+\lambda \frac{dz}{z}.
\end{equation}
We can define the formal flow of such a vector field as 
$$\exp(tv)=\sum_{n\ge0}\frac{t^n}{n!}\underbrace{v\cdot(\cdots (v}_{n\ \text{times}}\cdot z)\cdots)$$ 
where $v$ acts as a differential operator: $v\cdot g=f\partial_z\cdot g=f\cdot g'$.
For instance, we have $\exp(t v_\alpha)=\exp(t\alpha)z$ and
\begin{equation}\label{eq:formodexpvklambda}
\exp(tv_{k,\lambda})=z+tz^{k+1}+\left(\frac{k+1}{2}t^2-\lambda t\right)z^{2k+1}+o(z^{2k+1}).
\end{equation}
The description of formal diffeomorphisms up to conjugacy is similar (see \cite[ \S 1.3]{Frankpseudo}):

\begin{thm}\label{TH:formal}
Any $f\in \Difffor$ is conjugated to exactly one of the following models:
\begin{enumerate}
\item $f_0 (z)=az$ where $a\in{\C}^*$.
\item $f_0(z)=a \cdot \exp\left( v_{k,\lambda}\right)$ with $a^k =1$.
\end{enumerate}
\end{thm}

In the first case, $f$ is said to be (formally) \textit{linearizable}. In the second item, remark that $z\rightarrow az$ commutes with $\exp\left( v_{k,\lambda}\right)$ and that $f_0$ (hence $f$) has infinite order. We note that any element $f\in \Difffor$ is almost
contained in a formal flow: in case (2), the iterate $f_0^{\circ k}=\exp\left( k v_{k,\lambda}\right)$ is in a flow.
If we turn to abelian groups, the classification looks like that of flows:

\begin{thm}\label{TH:formalG}(\cite[\S 1.4]{Frankpseudo}) 
Any abelian subgroup $G \subset \Difffor$ is conjugated to a subgroup of one of the following models:
\begin{enumerate}
\item $\mb L :=\{ f(z) = a z\ ;\ a \in \C^* \}$.
\item $\mb E_{k,\lambda}:= \{a\cdot \exp\left(t v_{k,\lambda}\right)\ ;\ a^k=1,\ t\in\C \}$.
\end{enumerate}
These groups are respectively characterized as the group of elements of $\Difffor$ that 
commute with the respective models of $1$-form (\ref{eq:normformoneform}) or vector fields (\ref{eq:normformvectfield}).
Moreover, $\mb L $ (resp. $\mb E$) coincide with the centralizer in $\Difffor$ of any non torsion element
of $\mb L $ (resp. $\mb E$).
\end{thm}

\subsection{Foliations and closed $1$-forms}\label{sec:folclosed}

We immediately deduce:

\begin{cor}\label{cor:folformalclosedform}
Any foliation $\F\in\Folf$ can be defined by a formal closed meromorphic 
$1$-form $\omega$ on $U$ with polar divisor $(\omega)_\infty=(k+1)[C]$ for some $k\in\N$.
Moreover, the closed $1$-form $\omega$ defining $\F$ is unique up to a multiplicative constant
except when $\F$ is a fibration, i.e. admits a global first integral.
\end{cor}

\begin{proof}The holonomy $\rho_{\F}$ takes values in an abelian subgroup $G\subset\Difffor$.
By Theorem \ref{TH:formalG}, there exists a formal meromorphic $1$-form $\omega_0=h(z)dz$ invariant 
by $G$: if $\varphi_*G$ is in the model $\mb L$ (resp. $\mb E_{k,\lambda}$), 
then $\omega_0=\varphi^*\omega_{\alpha}$ (resp. $\varphi^*\omega_{k,\lambda}$).
There is a unique global and closed formal meromorphic $1$-form $\omega$ on $(U,C)$ which is invariant by $\F$,
i.e. $\omega\wedge\ df_i\equiv0$ for local formal first integrals, and coincide with $\omega_0$ in the holonomy coordinate:
$\omega\vert_{U_0}=h(f_0)df_0$ with notations of (\ref{eq:submdefolf}). Indeed, if one changes local first integrals $f_i:U_i\to\C$
such that all $\varphi_{ij}\in G$, then the the resulting form $f_i^*\omega_0$ remains unchanged and equal to $f_i^*\omega_0$.
We note that $\omega_0$ must have a pole at $z=0$, and therefore $\omega$ must have a pole along $C$.

Now, if there is another formal closed meromorphic $1$-form $\omega'$ also defining $\F$, 
then we must have $\omega'=f\omega$ for a formal meromorphic function $f$; after derivation we get
$df\wedge\omega=0$ so that $f$ is a global first integral for $\F$. If $f$ is not constant, 
then maybe changing to $1/f$, we get a non constant formal holomorphic map $f:U\to\C$ 
which is constant on the leaves, therefore a fibration.
\end{proof}

Conversely to Corollary \ref{cor:folformalclosedform}, we have:

\begin{thm}\label{thm:formalclosedformfol}
Let $\omega$ be a formal $1$-form whose polar divisor $(\omega)_\infty$ is supported by $C$
(or empty). Then
\begin{enumerate}
\item either $\F_\omega$ defines an element of $\Folf$,
\item or $\F_\omega$ is (regular) transversal to $C$.
\end{enumerate}
If $C$ is not the fiber of a fibration, then $\omega$ is closed.
\end{thm}

\begin{proof}
The divisor of $\omega$ can be written ${(\omega)}_0-{(\omega)}_\infty=E-(k+1) C$, where $E$ is an effective divisor
and $k\in\Z$. One can assume that no component of $E$ coincides with $C$. 
Viewing $\omega$ as a holomorphic section of $H^0(U,\omega^1_U((k+1)C))$, we get by restriction
to $C$ a section 
$$\omega\vert_C\in H^0(C,\omega^1_C\otimes N_C^{k+1}).$$
If it is a non trivial section, then it has no zero meaning that the twisted $1$-form $\omega$ is regular transversal to $C$.
If $\omega\vert_C\equiv0$, since the twisted $1$-form $\omega$ cannot vanish identically on $C$,
we deduce that $C$ is $\F_\omega$-invariant. In this latter case, we can consider
the residue 
$$\mbox{Res}_C \left(f^{-k-1}\omega\right)\in H^0(U,{{N}_C}^{\otimes k})$$ 
(where $f$ is a section of $\OO_U(C)$ vanishing on $C$) 
which must be a non trivial section, hence nowhere vanishing. In particular $E$ is trivial, and $\F_\omega$ is regular, 
but this time belonging to $\Folf$. 
In each case, $\F$ can be defined by a formal closed meromorphic $1$-form $\omega_0$ whose polar divisor $(\omega)_\infty$ is supported on $C$: in case (1) it follows from Corollary \ref{cor:folformalclosedform}, and in case (2)
we define $\omega_0=f^{-1}(dx)$ where $f:U\to C$ is the fibration and $dx$ is the holomorphic $1$-form on $C$.
If $C$ is not the fiber of a fibration, then there does not exist non constant formal holomorphic function on $U$:
such a function should be constant on $C$ and  consequently define a formal fibration, contradiction.
But $\omega=f\omega_0$ for some formal meromorphic function $f$; since divisors of $\omega$ and $\omega_0$
are supported on $C$, so is the divisor of $f$. We conclude that $f$ or $1/f$ is formal holomorphic and therefore constant.
\end{proof}

\subsection{Holonomy and periods of closed $1$-forms}\label{sec:holonomyperiods}

Consider the vector space $\closed$ of formal closed meromorphic $1$-forms whose polar divisor is (empty or) supported on $C$. 
We can define the periods of $\omega$ as a morphism $\pi_1(U\setminus C)\to(\C,+)$ obtained by integration of $\omega$ along paths. Note that this is well defined even if ${\omega}$ is only a formal $1$-form. Indeed, let $\pi : (\tilde U,E) \to (U,C)$ be the real analytic polar blow-up along $C$. The exceptional divisor is a $S^1$-bundle over $C$ with fiber over a point $p\in E$  parametrizing the ray on
a transversal to $C$ through $\pi(p)$. Given $\omega\in\closed$, one can integrate it to get a {\it primitive}
on the rays near $p$ taking the form $-\frac{\alpha}{ky^k} + \lambda \log y+c+o(y)$.
It is well defined by $\omega$ up to some integration constant. 
Following paths on $E$ we obtain a representation $\mathrm{per}_\omega : \pi_1(E) \to \mathbb C$. 
Since $E$ is a retraction of $U\setminus C$,  we have the sought representation. 

When $N_C$ is torsion, of order $m$ say, we can deduce a representation $\sigma_\omega : \pi_1(C) \to \mathbb C$ as follows. 
First consider the unitary representation $\rho_1: \pi_1(C) \to \mathbb C^*$ taking values in $m^{\text{th}}$ roots of unity.
Let $K:=\mbox{ker}\ {\rho}_1$ denotes the kernel: $K\subset\lattice$ has index $m$ in the lattice.
The  $S^1$-principal bundle $E$ 
comes naturally endowed with a flat connection with monodromy representation $\rho_1$ given by the unitary 
connection attached to $N_C$. This allow to lift loops $\gamma\in K$ as loops $\tilde\gamma\in\pi_1(E)$.
Define $\sigma(\gamma)=\mathrm{per}_\omega(\tilde\gamma)$ for all $\gamma\in K$. This can be extended as a unique
morphism $\sigma_\omega : \pi_1(C) \to \mathbb C$ on the whole lattice. This is no more a period mapping
since it is twisted by the unitary connection: for $\gamma\in \lattice\setminus K$, there is no lift $\tilde\gamma$ for which 
we would have $\sigma(\gamma)=\mathrm{per}_\omega(\tilde\gamma)$.

When $N_C$ is not torsion, there is no canonical way to deduce a morphism $\pi_1(C) \to \mathbb C$
from the period map.
However, as $(U,C)$ is a topologically trivial neighborhood, the fundamental group split as a product 
$\pi_1(U\setminus C)\simeq\pi_1(C)\times\Z$ and we can define 
$\sigma^\infty_\omega:\pi_1(C)\to\C$ as the restriction of the period map $\mathrm{per}_\omega$ to the first component. 
Note that $\sigma_\omega^\infty$ depends on the topological trivialization, except when $\omega$ has no residue along $C$.  Note also, in the torsion case, that the two defined morphisms $\sigma_\omega$ and $\sigma_\omega^\infty$
have no reasons to coincide unless $N_C$ is trivial.

One easily checks:

\begin{prop}\label{prop:holonomyperiods}
Let $\F$ and $\omega$ like in Corollary \ref{cor:folformalclosedform}
and let $\omega_0$ be the restriction of $\omega$ to a (formal) transversal $(T,z)\simeq(\C,0)$.
\begin{enumerate}
\item If $\omega_0=\frac{dz}{z}$, then $\varphi_\gamma(z)=e^{\sigma^\infty_\omega(\gamma)}z$.
\item If $\omega_0=\frac{dz}{z^{k+1}}+\lambda\frac{dz}{z}$, 
then $\varphi_\gamma(z)=a_\gamma\cdot\exp\left(\sigma_\omega(\gamma) v_{k,\lambda} \right)$
where $a_\gamma$ is the holonomy of the unitary (actually torsion of order dividing $k$) connection on $N_C$.
\end{enumerate}
\end{prop}

\begin{proof}It is a direct computation. \end{proof}

\begin{remark}
In item (1), $\sigma^\infty_\omega$ may depend on the trivialization, but $e^{\sigma^\infty_\omega}$ does not. Indeed, $\sigma^\infty_\omega$ is completely determined up to a multiple of $2i\pi$ corresponding to the monodromy of a primitive of $\omega_0$ around $C$. 
\end{remark}

To summarize, we have defined a period morphism 
\begin{equation}\label{eq:periodmorph}
\sigma^\infty:\closed\to \Hom(\lattice,\C)\ ;\ \omega\mapsto \sigma^\infty_\omega.
\end{equation}
depending on the choice of a topological trivialization of $(U,C)$, and when $N_C$ is torsion,
we have exhibited a canonical morphism 
\begin{equation}\label{eq:periodmorphK}
\sigma:\closed\to \Hom({\lattice},\C)\ ;\ \omega\mapsto \sigma_{\omega}.
\end{equation}
In case $C$ is not a fiber of a fibration,  we then have a natural inclusion $\Folf\subset\PP(\mc C)$.

\begin{cor}\label{cor:noperiodfibration}
Suppose that $N_C$ is torsion and let $\omega\in\closed$ non trivial. Then are equivalent 
\begin{itemize}
\item $\omega$ has no periods along $C$, i.e. $\sigma_{\omega}\equiv0$, 
\item $\F_\omega$ defines a fibration with fiber $C$.
\end{itemize}
Therefore, if $C$ is not the fiber of a fibration, then the period map $\sigma$ 
 is injective
and we have a natural inclusion $\Folf\subset\PP(\closed)$.
\end{cor}

\begin{proof} We follow notations of Theorem \ref{thm:formalclosedformfol} (and the end of its proof).
If we are in case (2), then $\omega$ coincides with a constant multiple of the closed $1$-form $\pi^{*}(dx)$ 
where $\pi:U\to C$ is the transverse fibration and $dx$ is the fixed abelian differential on $C$;
therefore, $\omega$
has non trivial periods (indeed, $\sigma_\omega$ coincides with the usual period map of the abelian differential $dx$). 
If we are in case (1), we have 
$\omega=f\omega_0$ where $\omega_0$ is the closed $1$-form constructed in 
Corollary \ref{cor:folformalclosedform}. After derivation, we get $df\wedge\omega_0=0$ which means that $f$
is a first integral of $\F_\omega$.  If $C$ is not the fiber of a fibration, then it implies again
 that $f$ is a constant and we can conclude that $\sigma_{\omega}$ is not identically zero, thanks to Proposition \ref{prop:holonomyperiods};
if on the contrary $C$ is a fiber of a fibration, then $f$ is non constant only in the case
where both $f$ and $\omega_0$ define the fibration (in which case $\sigma_{\omega}$ is easily seen to be trivial), otherwise we also conclude with Proposition \ref{prop:holonomyperiods}. Finally, if $C$ is not the fiber of a fibration,
then $\ker(\sigma)=\{0\}$; moreover, an element $\F\in\Folf$ is defined by a unique
closed $1$-form $\omega\in\closed$ up to a constant, so that it naturally defines an element 
of $\PP(\closed)$.
\end{proof}

\begin{remark}\label{R:infiniteorder}
Using the same type of argumentation, it is easy to see that $\sigma^\infty$ is injective whenever $N_C$ has infinite order.
\end{remark}

\section{Pencil of formal foliations}\label{S:Existence}

In this section, we aim at proving the following:

\begin{thm}\label{thm:feuilletagesformels} 
Assume that $N_C$ is torsion, but $C$ is not the fiber of a fibration.
Then the period morphism (\ref{eq:periodmorph}) defined in section \ref{sec:holonomyperiods}
$$\sigma:\closed\to \Hom({\lattice},\C)$$
is an isomorphism. 
Moreover, we can choose generators $\hat\omega_0, \hat\omega_\infty$ of $\mc C$ 
such that 
\begin{itemize}
\item  $\hat\omega_0$ has a pole of order $k+1>1$ along $C$, and $\hat\omega_0$ has zero period along $1\in\lattice$, i.e. $\sigma_{\hat\omega_0}(1)=0$;
\item $\hat\omega_\infty$ has a pole of order $0\le p+1<k+1$ along $C$.
\end{itemize}
These properties characterize $\hat\omega_0$ and $\hat\omega_\infty$ up to a constant. Moreover,  
the torsion of $N_C$ is dividing $p$ and $k$.
\end{thm}

Consider the pencil of foliations $(\F_t)$ defined by 
\begin{equation}\label{eq:pencilfol}
\F_t\ :\ \hat\omega_t:=\hat\omega_0+t\hat\omega_\infty=0,\ \ \ t\in\PP^1.
\end{equation}
They are all different $\F_t\not=\F_s$ because having non conjugated holonomy representations 
(see Proposition \ref{prop:holonomyperiods}). Finally, we have

\begin{cor}\label{cor:pencilfol}
If $N_C$ is torsion, but $C$ is not the fiber of a fibration, then 
\begin{itemize}
\item either $\Folf=\{\F_t\ ;\ t\in\PP^1\}$ and there is no fibration transversal to $C$;
\item or $\Folf=\{\F_t\ ;\ t\in\C\}$ and $\F_\infty$ is a fibration transversal to $C$.
\end{itemize}
Moreover, $\F_0$ and $\F_\infty$ are uniquely characterized by the following properties:
\begin{itemize}
\item $\F_0\in\Folf$ has torsion holonomy along $1\in\lattice$;
\item $\F_\infty\in\Folf$ has holonomy less tangent to its linear part than others, or is transversal to $C$.
\end{itemize}
\end{cor}

It might be interesting to compare with the non torsion case. When $N_C$ is not torsion, Arnold proved in \cite[Theorem 4.2.1]{Arnold} that the neighborhood $(U,C)$
is formally linearizable, i.e. formally equivalent to the neighborhood of the zero section in the total space of $N_C$.
In this situation, we have a one parameter family of foliations, namely those given by flat holomorphic connections on $N_C$
(the arguments stated below show that there are no other local foliations near the zero section). They are defined by a pencil of closed logarithmic
$1$-forms $\omega_0+t\omega_\infty$ where we can choose (see exact sequence (\ref{eq:exactseqflatlinear}))
\begin{itemize}
\item $\omega_0$ with residue $1$ on $C$ and purely imaginary periods along $C$ (i.e. the associated foliation $\F_0$ has unitary holonomy);
\item $\omega_\infty=dx$ to be the transverse fibration.
\end{itemize}

\subsection{Foliations in the non fibered case: existence and unicity of the pencil}\label{sec:FoliationNonTorsion}

We start following \cite{Ueda}, including the non torsion case. It should be mentioned that the existence of formal foliations has already been proved in \cite[Theorem A, p.3]{leaves} without restrictions on the genus of the embedded curve. However, the proof given here is simpler and provides further informations.

Select $(U_i)$ an open covering of some neighborhood of $C$. Let $\mc V=(V_i)$ the open covering of $C$ defined by $V_i:=U_i\cap C$. One can choose $(U_i)$ in such a way that $\mc V$ is an acyclic covering by disks. Denote by $\widehat{U_i}$ the formal completion of $U_i$ along $V_i$. Let ${y_i}\in {\mc O}(\widehat{U_i})$ some formal submersion such that $\{{y_i}=0\}=V_i$ and such that
$$a_{ij}y_j -{y_i}= y_i^2f_{ij}$$
where $a_{ij}\in Z^1(\mc V, \mc O^*_C)$ is a cocycle defining the normal bundle $N_C$ of the curve in $U$
and $f_{i,j}\in\mc O(\widehat{U_i}\cap \widehat{U_j})$.
The long exact sequence derived from 
$$0\to\C^*\to \mc O_C^*\stackrel{d\log}{\to}\Omega^1_C\to 0$$
gives in particular
\begin{equation}\label{eq:exactseqflatlinear}
0\to H^0(C,\Omega^1_C)\to H^1(C,\C^*)\to \mathrm{Pic}^0(C)\to 0.
\end{equation}
Therefore, we can choose $a_{ij}$ locally constant. Note that the cohomology class $[a_{ij}]\in H^1(C,\C^*)$ can be regarded as (the monodromy of) a flat connection attached to $N_C$.

If $N_C$ is not torsion, we can modify submersions $y_i$ in such a way that all $f_{ij}\equiv 0$
so that we get a global foliation locally defined by $dy_i=0$. Indeed, assume that we have
$$a_{ij}y_j -{y_i}= b_{ij}y_i^{n+1} +f_{ij}y_i^{n+2}$$
with $f_{ij}\in \mc O(\widehat{U_i}\cap \widehat{U_j})$ and $b_{ij}\in \mc O(V_i\cap V_j)$ not all zero for some $n\in\Z_{>0}$;
we already have a foliation up to the order $n$. Then we have

\begin{lemma}\label{admissible}
Under notations above,  we have $(b_{ij})\in Z^1(\mc V, {N_C}^{\otimes -n})$, i.e. satisfies the cocycle condition
$$b_{ij}+a_{ij}^{-n}b_{jk}+a_{ij}^{-n}a_{jk}^{-n}b_{ki}=0,\ \ \ \forall\ i,j,k.$$
Moreover, after \textbf{admissible transformation} of order $n+1$
$$\tilde y_i=y_i+b_i y_i^{n+1},\ \ \ b_i\in\mc O(\widehat{U_i})$$
we get the new cocycle
$$a_{ij}\tilde y_j -{\tilde y_i}=\underbrace{ \left[ b_{ij}-(b_i^0-a_{ij}^nb_j^0)\right]}_{\tilde b_{ij}}\tilde y_i^{n+1} +o(\tilde y_i^{n+1})$$
where $b_i^0:=b_i\vert_C$ is the restriction. 
In particular, if the class $[b_{ij}]\in H^1(C, {N_C}^{ -n})$ is zero, 
then we can make $b_{ij}=0$.
\end{lemma} 

\begin{proof}
This is a straightforward computation.
\end{proof}

Now, by Serre duality
\begin{equation}\label{e:vanishingH1}
H^1(C,N_C^{-n})\simeq H^0(C, N_C^{n})=\left\{\begin{matrix}
\C&\text{if}\ N_C^{\otimes n}=\mc O_C,\\
0& \mbox{if not.}\hfill \end{matrix}\right.
\end{equation}
When $N_C$ is not torsion, we have $H^1(C,N_C^{-n})=0$ for all $n>0$ and can successively kill the $(b_{ij})$ at each order.
We finally arrive at formal coordinates $y_i$ satisfying $y_i=a_{ij}y_j$. The corresponding formal foliation $\F$, 
locally defined by $\frac{dy_i}{y_i}=0$, is not convergent in general, but is convergent for generic $N_C$
(in the sense of Lebesgue measure on $\mathrm{Jac}(C)$) as proved by Arnol'd in \cite[Theorem 4.3.1]{Arnold}.
We note that the construction of $\F$ is unique once we have fixed the class of the cocycle $(a_{ij})$ in $H^1(C,\C^*)$. Indeed, consider two foliations $\F$,${\F}'$ together with the same linear holonomy (corresponding to the same $[a_{ij}]\in H^1(C,\C^*)$. Both holonomy representation being abelian, we deduce from theorem \ref{TH:formalG} that they are in fact linearizable. This allows to construct on the formal neighborhood of $C$ a function $f$ which locally expresses a the quotient of a first integral of $\F$ and ${\F}'$. Observe now that $f$ is necessarily constant (and consequently $\F={\F}'$), otherwise it would provide a global equation for $C$, contradicting the fact that $N_C$ has infinite order. The pencil of formal foliations described above can be then obtained by taking $\omega_0$ to be the unique logarithmic form with residue 1 on $C$ associated to the foliation in $\Folf$ with unitary holonomy and setting $\omega_\infty=\omega_0 -\omega_1$ where $\omega_1$ is a logarithmic form with residue 1 defining an element in $\Folf$ with non unitary holonomy. Thanks to Remark \ref{R:infiniteorder}, the formal $1$-form (without poles) $\omega_\infty$ has non trivial periods along $C$, and then defines a (necessarily unique) transverse fibration.

\subsection{Torsion case: fibration and Ueda class}\label{sec:Uedatype}
When $N_C$ is torsion, say of order $m$, then we can choose $a_{ij}^m=1$.
We can follow the same algorithm as in the non torsion case, but the process can stop 
each time $m$ divides $n$.
If not, i.e. if $(b_{ij})$ is zero as a cocycle at each step,
then Arnol'd proved in \cite{Arnold} that the process can be done in a convergent way,
so that we get holomorphic submersions $y_i$ satisfying $y_i=a_{ij}y_j$.
The functions $y_i^m$ patch together to define a global fibration $y:U\to\C$
and $mC$ is the zero fiber.

From now on, we assume the converse:
$$a_{ij}y_j -{y_i}= \frac{b_{ij}}{k}y_i^{k+1} +o(y_i^{k+1})$$
for some $k\equiv 0\ \mathrm{mod}\ m$ and $[b_{ij}]\in H^1(C, {N_C}^{\otimes -k})=H^1( C, \mc O_C)$
is not zero; equivalently
$$\frac{1}{{y_i}^k}-\frac{1}{{y_j}^k}= b_{ij} +o(1).$$
This cohomology class is defined by Ueda in \cite{Ueda} and is unique up to multiplication by a non zero scalar. 
It turns out to be a formal invariant which gives an obstruction to realize (formally)  the neighborhood of $C$ 
as the neighborhood of the zero section in the total space of the line bundle $N_C$. 
In Ueda's terminology, the neighborhoods with this obstruction are said to be ``of finite type''.
This allows us to define the \textbf{Ueda type} of the curve  
$$\utype(U,C):=k$$
(we set $\utype (C):=\infty$ when there is no formal obstruction). We note that we can choose the representative
$(b_{ij})$ of the Ueda class locally constant. Indeed, the long exact sequence derived from 
$$0\to\C\to \mc O_C\stackrel{d}{\to}\Omega^1_C\to 0$$
gives in particular
\begin{equation}\label{eq:exactseqflatSerre}
0\to H^0(C,\Omega^1_C)\to H^1(C,\C)\to H^1(C,\mc O_C)\to 0.
\end{equation}
The locally constant representative $(b_{ij})$ is unique up to periods %(or local system) 
of a holomorphic $1$-form. We choose one from now on.

\subsection{Torsion case: existence of foliations}\label{sec:ExistencePencilTorsion}

Now, let us start going on beyond Ueda obstruction: let us write
\begin{equation} \label{aijcocycle}
\frac{1}{{y_i}^k}-\frac{1}{{y_j}^k}= b_{ij} +c_{ij}{y_i}^n + o(y_i^n)
\end{equation}
for some $n\in\Z_{>0}$,
 with $[b_{ij}]\in H^1( C, \mc O_C)$ non zero and $c_{ij}\in{\mc O}(V_i\cap V_j)$.   
 
 \begin{lemma}\label{nullcoh}Under notations above,  $({c_{ij}})\in Z^1( \mc V, {N_C}^{\otimes -n})$
satisfies the cocycle condition. Moreover, if its class $[c_{ij}]\in H^1( C, {N_C}^{\otimes -n})$ is zero, 
then we can assume $c_{ij}=0$ after an admissible transformation of order $n+k +1$.
\end{lemma}

\begin{proof}The cocycle condition directly comes from considering the sum of (\ref{aijcocycle}) when 
$(i,j)$ runs over a cyclic permutation on three indices. 
Now, setting
$$\tilde y_i=y_i+c_i y_i^{n+k+1},\ \ \ c_i\in\mc O(U_i),$$
we get
$$\frac{1}{{\tilde y_i}^k}-\frac{1}{{\tilde y_j}^k}= b_{ij} +\underbrace{\left[c_{ij}-k(c_i^0-a_{ij}^{-n}c_j^0)\right]}_{\tilde c_{ij}}{\tilde y_i}^n + o(\tilde y_i^n)$$
where $c_i^0:=c_i\vert_C$ is the restriction to the curve.
If $[c_{ij}]\in H^1( C, {N_C}^{\otimes -n})$ is zero, then so is $[\frac{c_{ij}}{k}]$ and we can make $c_{ij}=0$ by conveniently choosing $c_i$.
\end{proof}

Using inductively Lemma \ref{nullcoh}, we arrive at
\begin{equation}\label{eq:inducnullcoh}
\frac{1}{{y_i}^k}-\frac{1}{{y_j}^k}= b_{ij} +c_{ij}{y_i}^{n} + o(y_i^n)
\end{equation}
with $n$ a multiple of $m=\ord(N_C)$. Since $H^1( C, \mc O_C)$ is one dimensional,
we can assume $(c_{ij})=\lambda\cdot (b_{ij})$ for some constant $\lambda\in\C$.

\begin{lemma}\label{lem:unicocykill} Under notations above, if $n\not=k$, then after an admissible transformation  
$\tilde y_i=y_i+c y_i^{n+1}$ with $c\in\C$,
one can assume that the class $[c_{ij}]\in H^1( C, {N_C}^{\otimes -n})$ is zero.
\end{lemma}

\begin{proof}From (\ref{eq:inducnullcoh}), we deduce
${y_j}^{n-k}=y_i^{n-k}+(n-k)b_{ij}y_i^n+o(y_i^{n})$
so that straightforward computation shows
$$\frac{1}{{\tilde y_i}^k}-\frac{1}{{\tilde y_j}^k}= b_{ij} + (c_{ij}+ck(n-k)b_{ij}){\tilde y_i}^n+o(\tilde y_i^n)
$$
so that we can fix $c$ to kill the class $[c_{ij}]\in H^1( C, \mc O_C)$.
\end{proof}

\begin{lemma}
Under notations above, if $n=k$ in (\ref{eq:inducnullcoh}) then 
$$\frac{1}{{y_i}^k}-\frac{1}{{y_j}^k}-\lambda\log{\frac{{y_i}^k}{{y_j}^k}}= b_{ij} + o({y_i}^k) $$
(we choose the principal determination of the logarithm near $1$).
\end{lemma}

\begin{proof}From (\ref{eq:inducnullcoh}), we deduce
${y_j}^k=y_i^k(1+b_{ij}y_i^k+o(y_i^{k}))$
so that 
$$-\lambda\log{\frac{{y_i}^k}{{y_j}^k}}=\lambda\log(1+b_{ij}y_i^k+o(y_i^{k}))=\lambda b_{ij}y_i^k+o(y_i^{k})$$
which proves the result.\end{proof}

\begin{lemma}
Assume that we have
\begin{equation}\label{aijcocyclelog}
\frac{1}{{y_i}^k}-\frac{1}{{y_j}^k}-\lambda\log{\frac{{y_i}^k}{{y_j}^k}}= b_{ij} + c_{ij}{y_i}^n+o(y_i^n)
\end{equation}
with $n>k$ and ${c_{ij}}\in{\mc O}(V_i\cap V_j)$. Then $({c_{ij}})\in Z^1( \mc V, {N_C}^{\otimes -n})$
satisfies the cocycle condition. Moreover, if its class $[c_{ij}]\in H^1( C, {N_C}^{\otimes -n})$ is zero, 
then we can assume $c_{ij}=0$ up to doing an admissible transformation of order $n+k +1$.
\end{lemma}

\begin{proof}
It is similar to that of Lemma \ref{nullcoh}. 
\end{proof}

\begin{lemma} Assume that the $y_i$'s satisfy the relation  (\ref{aijcocyclelog}) with $n>k$ multiple of $k$.
After an admissible transformation  $\tilde y_i=y_i+c y_i^{n+1}$ with $c\in\C$,
one can assume that the class $[c_{ij}]\in H^1( C, {N_C}^{\otimes -n})$ is zero.
\end{lemma}
 
\begin{proof}It is similar to that of Lemma \ref{lem:unicocykill}.
\end{proof}

When passing to the limit, one gets a collection of formal submersions $y_i$ satisfying  
\begin{equation}\label{eq:normallocalfirstintegrals}
\left(\frac{1}{{y_i}^k}-\lambda\log{{y_i}^k}\right)-\left(\frac{1}{{y_j}^k}-\lambda\log{{y_j}^k}\right)=b_{ij}.
\end{equation}
Finally, since we have chosen $(b_{ij})$ locally constant, we get after derivation a global closed meromorphic $1$-form 
$$\omega:=\frac{dy_i}{{y_i}^{k+1}}+\lambda \frac{d y_i}{y_i}.$$
In particular, the equation $\omega=0$ defines a regular formal foliation $\F$ having $C$ as a leaf;
a collection of regular formal first integrals is given by the submersions $y_i$.

\subsection{Proof of Theorem \ref{thm:feuilletagesformels} }\label{SS:infinitefoliations}
We now discuss the non unicity of our construction. 
Recall that we are assuming that $N_C$ is torsion, but $C$ is not 
the fiber of a fibration like in Theorem \ref{thm:feuilletagesformels}.
For each choice of $(b_{ij})\in H^1(C,\C)$ which is not zero in $H^1(C,\mc O_C)$,
we have constructed a formal closed meromorphic $1$-form $\omega\in\closed$ with a pole of order $k+1>1$ along $C$, $k$ being the Ueda type $\utype (U,C)$.
Moreover, formula (\ref{eq:normallocalfirstintegrals}) shows that the canonical morphism $\sigma_\omega$ is given by $\frac{1}{k}(b_{ij})$
through the natural isomorphism $H^1(C,\C)\simeq \Hom({\lattice},\C)$.
By this way, we have realized every morphism in $\Hom({\lattice},\C)$ except those given by periods of holomorphic $1$-forms on $C$.
From obvious dimensional reasons, the morphism $\sigma$ is thus surjective which, together with Corollary \ref{cor:noperiodfibration}, proves the first statement
of Theorem \ref{thm:feuilletagesformels}.

Note that the periods of holomorphic $1$-forms are realized by elements of $\closed$ having a lower order pole along $C$
(possibly free of poles). Indeed, let $\omega_\infty$ be generating this one dimensional subspace of $\closed$:
$$\ker\left(\closed\simeq \Hom({\lattice},\C)\simeq H^1(C,\C)\to H^1(C,\mc O_C)\right)=\C\cdot\omega_\infty.$$

One can write $\omega_\infty=\omega-\eta$ where $\omega,\eta\in\closed$ have a pole of order $k+1$ along $C$. Suppose that $\omega_\infty$ has also a pole of order $k+1$ and let $\F_\infty\in\Folf$ be the corresponding foliation. The holonomy representation of $\F_\infty$ has the form given by item (2) of Proposition \ref{prop:holonomyperiods} but this latter clearly implies that $\utype (U,C) > k$: a contradiction.

Finally, choose $\omega_0\in\closed\setminus\C\cdot\omega_\infty$ such that its period vanish 
along $m$ (or equivalently, any non zero multiple of $m$) $\in\lattice$.

 %%%%%%%%%%%%%%%%%%%%%%%%%%%%%%%%%%%%%%%%%%%%%%%

\section{Bifoliated classification}\label{S:BifClass}

A {\bf bifoliated neighborhood} is a $4$-uple $(U,C,\F,\G)$ where $(U,C)$ is a germ of neighborhood
of the elliptic curve $C$ and $\F,\G\in\Fol$ are distinct foliations having $C$ as a common leaf.
Another $4$-uple $(\tilde U,C,\tilde\F,\tilde\G)$ will be said equivalent to $(U,C,\F,\G)$ if there is an isomorphism
$\phi:(U,C)\to(\tilde U,C)$ between germs of neighborhoods satisfying (\ref{eq:equivNeighbId})
and conjugating the foliations:
$$\phi_*\F=\tilde\F\ \ \ \text{and}\ \ \ \phi_*\G=\tilde\G.$$
In this section, we consider the classification of bifoliated neighborhoods up to equivalence,
and show that it reduces to the classification of the pair of holonomy representations
$(\rho_\F,\rho_\G)$ up to some equivalence (see Theorem \ref{P:holaffstruct}).
All constructions and results will be valid and usefull both in the formal and analytic settings.
In fact, it will be applied to the canonical pair of formal foliations $(\F_0,\F_\infty)$ to deduce
the formal classification of neighborhoods $(U,C)$. We will also make use of the analytic setting,
not necessarily with the canonical pair, to 
construct huge families of analytic neighborhoods that are all formally equivalent, but pairwise not analytically equivalent.
Even in the non torsion case, our construction provides new examples
with formal but divergent fibration transversal to $C$. For simplicity, we will work in the analytic 
setting; we will just mention its formal counterpart in the main statements, the proof being exactly the same.

\subsection{Basic invariants}\label{sec:BifBasicInv}
A first invariant is given by the {\bf holonomy representation} 
of the common leaf $C$ for both foliations
$$\rho_{\F},\rho_{\G}\ :\ \pi_1(C,p)\to\Diff$$
each representation being defined up to conjugacy in $\Diff$
(see section \ref{sec:folhol}).
In the sequel, it will be convenient to define holonomy on a transversal curve $(T,y)$:
$T\subset(U,C)$ is a germ of curve at a point $p\in C$ transversal to $C$
and $y:(C,p)\to(\C,0)$ a coordinate. The holonomy of $\mathcal F$ on the transversal  $(T,y)$
is defined as follows. Let $f:(U,p)\to(\C,0)$ be the unique first integral for $\F$ whose restriction 
to $T$ coincides with the coordinate $y$: then define the holonomy of $\F$ as the monodromy of $f$
by setting $f_0=f$ in the construction of section \ref{sec:folhol}.
This allow us to compute and compare if necessary the two holonomy representations $\rho_{\F},\rho_{\G}$ on 
a same transversal $(T,y)$.

A second invariant is the {\bf order of tangency}  $\Tang(\F,\G)$ between $\F$ and $\G$ along $C$:
given non vanishing $1$-forms $\alpha$ and $\beta$ at the neighborhood of a point $p\in C$
defining the two foliations $\F$ and $\G$ respectively, then $\alpha\wedge\beta$ vanishes along $C$; 
being locally constant, the vanishing order, denoted by $\Tang(\F,\G)$ does not depend on the choice of $\alpha$ and $\beta$, is globally well-defined and is a positive integer. Let $k\in\N$ be such that  $\Tang(\F,\G)=k+1$.
The next proposition shows that this order of tangency can be detected from the knowledge of holonomy representations.

\begin{prop}\label{P:Tanghol}
Let $\F,\G\in\Fol$ and consider the corresponding holonomy representations $\rho_\F, \rho_\G$ evaluated on the same transversal $(T,y)$. Then the following properties are equivalent.
\begin{enumerate}
\item $\Tang(\F,\G)=k+1.$
\item $\rho_\F$ and $\rho_\G$ coincide to order $k$ in the variable $y$ but differ at order $k+1$.
\end{enumerate}
In particular, two foliations having the same holonomy must coincide.

\end{prop}

\begin{proof}
Assume that $\Tang(\F,\G)=k+1$. This implies that $\F$ and $\G$ coincide on restriction to the $k^{th}$ infinitesimal 
neighborhood $C_{(k)}= \mbox{Spec}\ {\mc O}_U /{{\mc I}_C}^{k+1}$; in particular, $\rho_\F$ and $\rho_\G$ coincide 
up to order $k$. Let $f_p, g_p$ two respective first integrals of $\F$ and $\G$ such that $\{f_p=0\}=\{g_p=0\}$ 
are reduced equation of $C$ and such that $f_p=g_p$ on $C_{(k)}$. 
From the hypothesis on $\Tang(\F,\G)$, this means that $f_p= g_p +a_p {g_p}^{k+1}$ where $a_p$ is a formal function
whose restriction  ${a_p}_{|C}$ to $C$ is a non constant holomorphic function. Let ${f_p}', {g_p}'$ be sharing the same 
properties. Hence, they are related to $f_p$ and $g_p$ by ${f_p}'= \phi (f_p)$, ${g_p}'=\psi (g_p)$ 
where $\phi,\psi\in\Diff$ and $\phi (y)=\psi (y)\  \mbox{mod}\ y^{k+1}$. It is then easy to verify that ${f_p}'= {g_p}' +{a_p}' {g_p}^{k+1}$ whith ${{a_p}'}_{|C}=\tau\circ {a_p}_{|C}$ where $\tau$ is an affine transformation of the complex line. 
As a byproduct, the collection of ${a_p}_{|C}$, when varying the base point $p$, define an (maybe branched) 
\textit{affine structure} on $C$ intrinsically attached to the pair $(\F,\G)$. Now, let us fix $p_0\in C$,  
pick a germ of transversal $(T,y)$ at $p_0$ and let $f,g$ the local first integrals of $\F$ and $\G$ which coincide 
with $y$ on $T$. Let $a$ such that  $f= g+ag^{k+1}$. Note that $a(p_0)=0$. Let $b_\gamma$ the germ at $ (C,p_0)$ obtained by analytic continuation (well defined by the affine structure) of $a_{|C}$ along $\gamma\in\pi_1(C, p_0)$. 
Clearly, the holonomy representations of $\F$ and $\G$ coincide up to order $k+1$ if and only if $b_\gamma (p_0)=0$ 
for every loop $\gamma$. In this case, there exists on $C$ a multivaluate non constant function $\xi$ locally holomorphic 
with multiplicative monodromy. This gives a contradition with the residue theorem applied to the logarithmic differential 
$\frac{d\xi}{\xi}$.
\end{proof} 

\begin{remark}The statement of Proposition \ref{P:Tanghol} remains valid when considering more generally $n+1$-dimensional bifoliated neighborhoods of $n$-dimensional compact K\"ahler manifolds (use the same arguments).
\end{remark}

\begin{remark}The affine structure exhibited in the proof will reappear in the main Theorem \ref{ThmBifoliatedNeighborhood_k>0} of this section. 
\end{remark}

\subsection{First case: $\Tang(\F,\G)=1$}\label{SS:tan=1}

There is the following {\bf constraint} on the holonomy representations. The foliation $\F$
defines a partial 
connection 
$\nabla_{\F}$ on the normal bundle $N_{\F}$
called {\bf Bott partial connection} whose restriction to the leaf $C$ is a flat connection on the normal bundle $N_C\simeq {N_\F}_{|C}$. The monodromy representation of ${\nabla_{\F}}_{|C}$ is given 
by the linear part of the holonomy. 
Comparing 
$\nabla_{\F}$ and $\nabla_{\G}$ on $N_C$, we get that there exists
a holomorphic $1$-form $\omega$ on $C$, namely $\omega={\nabla_{\F}}_{|C}-{\nabla_{\G}}_{|C}$ such that 
\begin{equation}\label{CompatibilityDiscrepency}
\mathrm{lin}\left(\rho_{\G}(\gamma)\right)=\exp\left({\int_\gamma\omega}\right)\cdot\mathrm{lin}\left(\rho_{\F}(\gamma)\right)
\end{equation}
for all $\gamma\in\pi_1(C,p)$. As we will see, we have
$$\Tang(\F,\G)=1\ \ \ \Leftrightarrow \ \ \ \omega\not=0$$
and in this case $\F$ and $\G$ are transversal outside $C$.

\begin{thm}\label{ThmBifoliatedNeighborhood} Case $\Tang(\F,\G)=1$.
\begin{itemize}
\item {\bf Classification} Two analytic/formal bifoliated neighborhoods $(U,C,\F,\G)$  and $(\tilde U,C, \tilde{\F},\tilde{\G})$ (both with $k=0$) are analytically/formally equivalent if, and only if, there 
exist analytic/formal diffeomorphisms $\phi,\psi\in\Diff / \Difffor$ such that for all $\gamma\in\pi_1(C,p)$
\begin{equation}\label{eq:Bifoliatedk=0}
 \left\{\begin{matrix}\phi\circ\rho_{\tilde{\F}}(\gamma)=\rho_{{\F}}(\gamma)\circ\phi\\ \psi\circ\rho_{\tilde{\G}}(\gamma)=\rho_{{\G}}(\gamma)\circ\psi
\end{matrix}\right.
\end{equation}
\item {\bf Realization} Given two representations $\rho_{\F},\rho_{\G}$ satisfying the compatibility rule (\ref{CompatibilityDiscrepency}) for a non zero $1$-form $\omega$ on $C$, there is a unique (up to isomorphism) bifoliated neighborhood $(U\supset C,\F,\G)$
realizing these invariants with $k=0$.
\end{itemize}
\end{thm}
\begin{remark}
In the statement of Theorem \ref{ThmBifoliatedNeighborhood}, we do not demand that the holonomy pairs are evaluated with respect  to the same transversal, contrary to the higher tangency case (see Theorem \ref{ThmBifoliatedNeighborhood_k>0}). Moreover, as one of the conjugated representations, says $\rho_\G,\rho_{\tilde\G}$ is linearizable (see proof of Theorem \ref{ThmBifoliatedNeighborhood}), we can impose to the diffeomorphisms $\phi,\psi$ involved in \ref{eq:Bifoliatedk=0} to have the same linear part. This corresponds to the statement of Theorem \ref{ThmBifoliated} of the introduction for $k=0$.
\end{remark}

\begin{remark} One can more generally consider bifoliated neighborhoods of any compact curve $C$ regardless of its genus $g=g(C)$. When $g>1$, the $1$-form $\omega$ vanishes at some points
meaning that we have other tangencies between $\F$ and $\G$. This gives rise
to additional invariants, even locally (see \cite{Olivier}).
\end{remark}

Before giving the proof of theorem \ref{ThmBifoliatedNeighborhood}, let us start by some useful remarks. 

Let $\F,\G$ be foliations fullfilling the hypothesis of \ref{SS:tan=1}.  In local coordinates $(x,y)$ at $p$, where $C=\{y=0\}$,
we can choose the $1$-forms defining $\F$ and $\G$ uniquely as follows
$$\left\{\begin{matrix}\alpha=dy+ya(x,y)dx\hfill \\ \beta= dy+yb(x,y)dx\end{matrix}\right.
\ \ \ \text{with}\ a,b\in\C\{x,y\}.$$
Then, the holomorphic $1$-form defined on $C$ by 
$$\omega:=[b(x,0)-a(x,0)]dx=\mbox{Res}_C (\dfrac{\beta\wedge\alpha}{y^2})$$ 
is not identically vanishing, and does not depend on the choice
of coordinates $(x,y)$. This is an invariant of the pair of foliations measuring
the discrepancy between them at the first order along $C$. 
We thus get a global non zero holomorphic $1$-form $\omega$ on $C$, 
which is precisely the $1$-form involved in (\ref{CompatibilityDiscrepency}). Indeed, one has ${\nabla_{\F}}_{|C}=d-a(x,0)dx$ and ${\nabla_{\G}}_{|C}=d-b(x,0)dx$.
Since $C$ is an elliptic curve, it follows that $\omega$ does not vanish at all, and $\F$ is therefore
transversal to $\G$ outside of $C$.
In order to prove Theorem \ref{ThmBifoliatedNeighborhood}, we need 
the following local classification result, which is a version of \cite[Lemma 5]{FrankRamis} for functions.

\begin{lemma}\label{LemNormFormBifol}
Let $f,g\in\C\{x,y\}$ be two local reduced equations for $C=\{y=0\}$
and assume that the zero divisor of $df\wedge dg$ is $C$ (hence also reduced).
Then, there is a unique system of local coordinates $(\tilde x,\tilde y)$
satisfying
$$f=\tilde y\ \ \  g=e^{u(\tilde x)}\tilde y\ \ \ \text{and}\ \ \ (\tilde x,\tilde y)\vert_C=(x,0)$$
for a submersive function $u$ unique modulo $2i\pi$.
\end{lemma}

The invariant $u(x)$ is given by the limit of the ratio of the two functions:
$$e^{u(x)}\vert_C:=\frac{g(x,y)}{f(x,y)}\vert_{y=0}.$$
For the corresponding foliations, the invariant $1$-form discussed above is $\omega:=du$.
As noticed above, it can also be interpreted as the Poincar\'e residue of the $2$-form 
$\frac{dg}{g}\wedge\frac{df}{f}$ along $C$.

In the Lemma, we have only considered changes of coordinates fixing $C$ pointwise
since this is what we need for the proof of Theorem \ref{ThmBifoliatedNeighborhood}.
On the other hand, if we just ask that the change of coordinates preserve $C$
globally, then we arrive at the unique normal form
$$f=\tilde y\ \ \  g=(c+\tilde x)\tilde y$$
where $c:=e^{u(0)}$ is a non zero constant.

\begin{proof}Let us first set $\tilde y:=f$ and expand 
$$g=g_1(x)\tilde y+g_2(x)\tilde y^2+g_3(x)\tilde y^3+\cdots$$
Then $df\wedge dg=-(\sum_{n>0}g_n'(x)\tilde y^n) dx\wedge d\tilde y$ and we get $g_1'(x)\not=0$;
in fact, $g_1(x)=u(x)$.
We can write $g_1(x)=c+\varphi(x)$ with $c=g_1(0)\not=0$ and $\varphi(x)$ a local coordinate.
Then $g=(c+\phi(x,\tilde y))\tilde y$ where
$$\phi(x,\tilde y)=\varphi(x)+\sum_{n>0}g_{n+1}(x)\tilde y^n=\frac{g}{\tilde y}-c.$$
Clearly, $(\tilde x,\tilde y):=(\phi(x,\tilde y),\tilde y)$ is the unique system of coordinates putting $g$ 
into the normal form
$g=(c+\tilde x)\tilde y$. Finally, the coordinate $\varphi^{-1}\circ\phi$ induces the identity on $C$
and conjugates $g$ to $\exp(\log(c+\varphi(\tilde x)))\tilde y$.
\end{proof}

\begin{proof}[Proof of Theorem \ref{ThmBifoliatedNeighborhood}] Given $(U,C,\F,\G)$ and $(\tilde U,C,\tilde{\F},\tilde{\G})$
together with $k=0$ and satisfying \ref{eq:Bifoliatedk=0}. Choose a point $p\in C\subset U,\tilde U$. Maybe after changing first integrals by left composition in $\Diff$:
$(\tilde f,\tilde g)\to(\phi\circ\tilde f,\psi\circ\tilde g)$, we may assume that 
$\rho_{\F}= \rho_{\tilde{\F}}$ and $\rho_{\G}= \rho_{\tilde{\G}}$.

Let $\omega$ and $\tilde\omega$ the discrepancy $1$-forms associated to the pairs 
$(\F, \G)$ and $( \tilde{\F},\tilde{\G})$ respectively.
Since holonomy representations are the same for the two pairs, we must have 
$$\exp\left({\int_\gamma\omega}\right)=\exp\left({\int_\gamma\tilde\omega}\right)$$
for all $\gamma\in\pi_1(C,p)$. It follows that $\omega=\tilde\omega$.

One of the two representations, say  $\rho_{\G}$, must be non unitary.
Therefore, according to Koenigs theorem, it is analytically conjugated to its linear part (see Theorem \ref{thm:Koenigs}). Changing again if necessary $g$ and $\tilde g$,
we can assume that $\rho_{\G}=\rho_{\tilde{\G}}$ is linear. In particular, we can now multiply
$g$ and $\tilde g$ by constants without modifying their holonomy; therefore, we can now assume
that $g/f$ and $\tilde g/\tilde f$ have common limit $1$ at the point $p$.
By consequence, these ratios have the same limit $\exp(u(x))$ on $C$.

Lemma \ref{LemNormFormBifol} says that there is a unique germ of diffeomorphism $\Phi:(U,p)\to(\tilde U,p)$ fixing $C$ pointwise conjugating the two pairs of first integrals.
By analytic continuation of the first integrals, the existence and unicity of $\Phi$ at each point implies the analytic continuation
principle for $\Phi$ along any path in $C$. Having chosen $(\tilde f,\tilde g)$ with the same monodromy as $(f,g)$, we get a globally well-defined
diffeomorphism $\Phi:(U,C)\to(\tilde U,C)$.

The existence part of Theorem \ref{ThmBifoliatedNeighborhood} can be done by quotient,
similarly to the usual construction by suspension when we want to realize a single representation.
The difference is that, instead of preserving a fibration, we preserve a second foliation.
Consider on the universal cover $\C_x\to C$ the trivial line bundle
$\C_x\times\C_y$ together with models 
$$f=y\ \ \ \text{and}\ \ \ g=e^{\theta x}\cdot y$$
where $\omega=\theta dx$ is the $1$-form of the Theorem. For each period $\gamma$,
we want to construct a diffeomorphism $\phi_\gamma$ at the neighborhood of $\tilde C=\{y=0\}\simeq\C_x$ such that
\begin{itemize}
\item $\phi_\gamma$ is preserving $\tilde C$ and inducing the translation by $\gamma$ on it,
\item $\phi_\gamma$ is conjugating $(\rho_{\F}(\gamma)\circ f,\rho_{\G}(\gamma)\circ g)$ with $(f,g)$.
\end{itemize}
Note that $\rho_{\F}(\gamma)$ is a (germ of) diffeomorphism fixing $0$ so that, composed with $f$,
we still get a first integral of the associated foliation at the neighborhood of $\tilde C$, vanishing on $\tilde C$. We deduce the existence of $\phi_\gamma$ from Lemma \ref{LemNormFormBifol} locally at any point.
Indeed, we just have to check that the infinitesimal invariant $e^u$ is the same for
\begin{itemize}
\item $(f,g)$ at $(x+\gamma,0)$ and
\item $(\rho_{\F}(\gamma)\circ f,\rho_{\G}(\gamma)\circ g)$ at $(x,0)$
\end{itemize}
which straightly follows from assumptions. By uniqueness,
the local diffeomorphisms provided by Lemma \ref{LemNormFormBifol} patch together into a global diffeomorphism $\phi_\gamma$.
The map $\gamma\mapsto\phi_\gamma$ defines a group action on the germ of neighborhood of $\tilde C$.
On the quotient, we get a neighborhood $U\supset C$ equipped with two foliations $\F$ and $\G$
having the expected holonomy. Note that invariants $k=0$ and $\omega$ are actually determined by the monodromy.
\end{proof}

\begin{example}\label{ExampleLinear}For linear representations
$$\left\{\begin{matrix}
\rho_{\F}(\gamma)\ :\ y\mapsto a_\gamma y\\
\rho_{\G}(\gamma)\ :\ y\mapsto b_\gamma y
\end{matrix}\right.$$
we get the foliations associated to connections $\nabla_{\F},\nabla_{\G}$
of a (degree zero) line bundle $L\to C$. In this case, the $\phi_\gamma$ constructed in the proof above
are just given by $(x,y)\mapsto (x+\gamma,a_\gamma y)$. Indeed,
$$\phi^*(f,g)=(a_\gamma y,e^{\theta(x+\gamma)}a_\gamma y)=(a_\gamma f,b_\gamma g)$$
(we must have $b_\gamma=e^{\theta\gamma}a_\gamma$ for any period $\gamma$).
In particular, in this case, the fibration defined by the $x$-variable is preserved.
\end{example}

\subsection{Second case: $\Tang(\F,\G)=k+1,\  k>0$ }

As before, we only deal with \textit{analytic} bifoliated neighborhoods in the proofs and leave  the reader to adapt the argument  in the formal setting..
We can choose local coordinates $(x,y)$ so that $y$ is a first integral for $\F$,
i.e. so that 
$$\left\{\begin{matrix}\alpha=dy\hfill \\ \beta= dy+y^{k+1}h(x,y)dx\end{matrix}\right.
\ \ \ \text{with}\ h(x,0)\not\equiv 0.$$ are defining forms for $\F$ and $\G$.
Then, the holomorphic $1$-form $\omega=h(x,0)dx$ is well-defined up to multiplication by a constant.
If we do this with a uniformizing coordinate $x$ on $C$, then the logarithmic derivative of $h(x,0)$ 
defines a logarithmic $1$-form $\eta=\frac{dh}{h}\vert_C$ with only positive residues (zeroes of $h(x,0)$).
By Residue Theorem, $\eta=\theta\cdot dx$ for a constant $\theta\in\C$, 
and 
$$\omega=
c\exp(\theta x)dx$$
for a constant $c\in\C^*$. 
In particular, $\omega$ does not vanish and $\F$ is transversal to $\G$ outside $C$.
The invariant found here is not a global holomorphic $1$-form in general, 
but a global affine structure on $C$: an affine coordinate is given by integration of  $\omega$ (regarded as a holomorphic form on the universal covering)
\begin{equation}\label{Eq:Formuledux}
u(x)=\int_0^x\omega=\left\{\begin{matrix}
c\left(e^{\theta x}-1\right)\ \ \ \text{when}\ \ \theta\not=0\\
cx\hfill \text{when}\ \ \theta=0
\end{matrix}\right.\end{equation}
with monodromy given by
$$u(x+\gamma)=\left\{\begin{matrix}
e^{\theta\gamma}u(x)+c\left(e^{\theta\gamma}-1\right)\\
u(x)+c\gamma
\end{matrix}\right.
$$
(here, we have imposed $u(0)=0$ which is convenient for the sequel).

The version of Lemma \ref{LemNormFormBifol} when $k>0$ reads as follows:

\begin{lemma}\label{LemNormFormBifolBis}
Let $f,g\in\C\{x,y\}$ be two local reduced equations for $C=\{y=0\}$
and assume that the zero divisor of $df\wedge dg$ is $(k+1)[C]$.
Then, there is a unique system of local coordinates $(\tilde x,\tilde y)$
satisfying
$$f=\tilde y\ \ \  g=P(\tilde y) + u(\tilde x)\tilde y^{k+1}\ \ \ \text{and}\ \ \ (\tilde x,\tilde y)\vert_C=(x,0)$$
for unique non vanishing function $u$ on $C$ and degree $\leq k$ polynomial coordinate $P$ 
($P(0)=0$ and $P'(0)\neq 0$).
\end{lemma}

We note that $P$ is characterized by $g=P(f)+O(f^{k+1})$ and the function $u(x)$ on $C$
is given by the restriction of $\frac{g-P(f)}{f^{k+1}}$. The invariant for the pair of corresponding foliations 
is given by the Poincar\'e residue $\omega:=du$ of the $2$-form 
$\frac{dg\wedge df}{f^{k+2}}$ along $C$. In other words, local invariants for the pair $(f,g)$ are given by 
the polynomial $P$ and the function $u$ on $C$ while local invariants for the pair of foliations $(\F,\G)$ is given by $\C\cdot du$ (i.e. $du$ up to a constant).

\begin{proof} By local change of coordinate, we can first set $f=\tilde y$, and we get
$$g=P(\tilde y)+\tilde y^{k+1}(u(x)+\tilde yh(x,\tilde y))$$
where $P=\sum\limits_{n=1}^ka_n\tilde y^n$ and $a_1\not=0$.
Setting $\tilde x=u(x)+\tilde y h(x,\tilde y)-u(0)$
we get the sharper normal form
$$f=\tilde y\ \ \ \text{and}\ \ \ g=P(\tilde y)+\tilde y^{k+1}(u(0)+\tilde x).$$
If we restrict to $x$-change of coordinates inducing the identity along $y=0$
we get the normal form of the statement by setting $u(\tilde x)=u(x)+\tilde y h(x,\tilde y)$.
\end{proof}

As can be checked from the normal form, are equivalent:
\begin{itemize}
\item on any curve $\Gamma$ transversal to $C$ at $p=(0,0)$, the restrictions of $f$ and $g$ coincide 
up to the order $k+1$
(i.e. $f-g_{|C}=o(f^{k+1})$);
\item there exists a curve $\Gamma$ transversal to $C$ at $p$ in restriction to which $f\equiv g$
(and this curve is unique);
\item $P(y)=y$ and $u(0)=0$.
\end{itemize}
In this case, we say that $f$ and $g$ coincide up to the order $k+1$ at the point $p$. 
Given a pair of foliations $(\F,\G)$ satisfying $\Tang(\F,\G)=k+1$,
we can always find a pair $(f,g)$ of first integrals that coincide up to the order $k+1$. Note also that if $f,g$ coincide at order $k+1$ at $p$, then for $\phi,\psi\in\Diff$, the new pair $\phi\circ f$,$\psi\circ g$ coincide at order $k+1$ at $p$ if and only if $\phi (y)=\psi (y)\ \mbox{mod}\ y^{k+2}$. 

At another point of $C$, the function $u$ will be non zero in general, and we cannot find 
a transversal curve in restriction to which $f\equiv g$, but we still have $P(y)=y$ 
by analytic continuation and $f$ still coincides with $g$ up to the order $k$ on any transversal curve.

In order to settle the relationship between the affine structure previously defined
and the holonomy representations of the two foliations $\F$ and $\G$,
let us fix a point $p\in C$ and two local respective first integrals $f$ and $g$ 
that coincide on a transversal curve $\Gamma$. Let us compute the holonomy
representations of the two foliations with respect to these first integrals. 
After analytic continuation along a loop $\gamma$ based at $p$,  we will have on one hand 
$$f^\gamma=\varphi_{\F}^\gamma\circ f\ \ \ \text{and}\ \ \ 
g^\gamma=\varphi_{\G}^\gamma\circ g.$$
On the other hand, we have 
\begin{equation}\label{eq:CompatibilityOrderkFirstInt}
g=f+u(x)f^{k+1}\ \text{mod}\ f^{k+2}\ \ \ \text{with}\ \ \ u(x)=
\left\{\begin{matrix}
c\left(\exp(\theta x)-1\right)\\
cx \hfill\text{(if }\theta=0\text{)}
\end{matrix}\right.
\end{equation}
and, after analytic continuation along $\gamma$, we get (in restriction to $\Gamma$)
$$g^\gamma=\phi_\gamma\circ f^\gamma\ \text{mod}\ f^{k+2}\ \ \ \text{with}\ \ \ \phi_\gamma(y)=y+u(\gamma) y^{k+1}.$$
To summarize, we get
\begin{equation}\label{eq:CompatibilityOrderk}
\varphi_{\G}^\gamma=\phi_\gamma\circ\varphi_{\F}^\gamma\ \text{mod}\ y^{k+2}\ \ \ \text{for all}\ \gamma\in\pi_1(C,p)
\end{equation}
$$\text{with}\ \ \ 
\phi_\gamma(y)=\left\{\begin{matrix}
y+c(e^{\theta\gamma}-1) y^{k+1}\ \ \  \text{if}\ \theta\not=0\\
y+c\gamma y^{k+1}\hfill \text{if}\ \theta=0
\end{matrix}\right.$$

In particular this illustrates the general principe given in \ref{P:Tanghol} according to which the holonomies of both foliations coincide at order $k$ but differ at order $k+1$. 
The above equalities give restrictions on the way they can differ at the order $k+1$.  
Mind that $\gamma\mapsto\phi_\gamma$ is not a group morphism
(even after truncating coefficients of order $\ge k+2$).

The linear part $\mathrm{lin}(\rho_{\F})=\mathrm{lin}(\rho_{\G})$
of the holonomy of the two foliations
is related to the linear part $\gamma\mapsto e^{\theta\gamma}$ of the monodromy 
of the affine structure as follows. 

\begin{prop}\label{P:holaffstruct}If we set $\rho_{\F}(\gamma)=a_\gamma y+\cdots$,
then we have
\begin{equation}
a_\gamma^{-k}=e^{\theta\gamma}.
\end{equation}
In particular,  $\theta$ is completely determined by the linear part of the holonomy and the normal bundle $N_C$ is necessarily torsion of order dividing $k$.
Moreover, if $\mathrm{lin} (\rho_{\F})$ is unitary, then it is torsion
and the affine structure is actually a translation structure, i.e $\theta=0$.
\end{prop}

\begin{proof}From the formula (\ref{Eq:Formuledux}), we get 
$$\frac{u(2\gamma)}{u(\gamma)}-1=e^{\theta\gamma}$$
for all $\gamma\in\pi_1(C,p)$. On the other hand, comparing (\ref{eq:CompatibilityOrderk}) for 
$\phi_{\gamma}$ and $\phi_{2\gamma}$ with linear part $a_\gamma$ for $\rho_{\F}$ and $\rho_{\G}$ yields
$$\frac{u(2\gamma)}{u(\gamma)}-1=a_\gamma^{-k}$$
whence the relation. In particular, this shows that $N_C^{\otimes{k}}$ comes equipped with a connection, namely ${\nabla_\F}_{|C}^{\otimes {k}}$, whose monodromy is given by the exponential of the periods of the abelian differential $-\theta dx$. This implies that  $N_C^{\otimes{k}}\simeq {\mathcal O}_C$ and that the monodromy (given by  ${\mathrm{lin} (\rho_{\F})}^{\otimes{k}}$) is not unitary, unless $\theta$ vanishes identically. This concludes the proof.
\end{proof}

In order to state our result, we associate to any bifoliated neighborhood 
$(U,C,\F,\G)$ the holonomy representations $(\rho_{\F},\rho_{\G})$ computed from 
a pair $(f,g)$ of first integrals \textbf{that coincide on a given transversal at $p\in C$}.

\begin{thm}\label{ThmBifoliatedNeighborhood_k>0} Case $\Tang(\F,\G) =k+1 ,k>0$.
\begin{itemize}
\item {\bf Classification} Two analytic/formal bifoliated neighborhoods $(U,C,\F,\G)$  and $(\tilde U,C, \tilde{\F},\tilde{\G})$ (with the same $k>0$) are analytically/formally equivalent if, and only if, there 
exists analytic/formal diffeomorphisms 
$\phi,\psi\in\Diff / \Difffor$ \text{such that} for all $\gamma\in\pi_1(C,p)$,
\begin{equation}\label{eq:Bifoliatedk>0}
\left\{\begin{matrix}\phi\circ\rho_{\tilde\F}(\gamma)=\rho_{{\F}}(\gamma)\circ\phi\\ \psi\circ\rho_{\tilde{\G}}(\gamma)=\rho_{{\G}}(\gamma)\circ\psi
\end{matrix}\right.\ \ \ 
\mbox{with}\ \phi=\psi\ \text{mod}\ y^{k+2}.
\end{equation}
\item {\bf Realization} Given two representations $\rho_{\F},\rho_{\G}$ satisfying the compatibility rule (\ref{eq:CompatibilityOrderk}) for an affine structure defined by a chart $u(x)$ on $C$, 
there is a unique (up to isomorphism) bifoliated neighborhood $(U,C,\F,\G)$
realizing these invariants (with tangency divisor $(k+1)[C]$).
\end{itemize}
\end{thm}

\begin{proof}
{\bf Classification part.} Let $(\F,\G)$,$(\tilde{\F},\tilde{\G})$ two pairs of foliations satifying the property \ref{eq:Bifoliatedk>0}. After changing first integrals 
$(\tilde f,\tilde g)$ by $(\phi\circ\tilde f,\psi\circ\tilde g)$, we may assume that 
$\rho_{\F}= \rho_{\tilde{\F}}$ and $\rho_{\G}= \rho_{\tilde{\G}}$.
One easily checks that the affine coordinate $u(x)$ is determined by its values $u(\gamma)$ when 
$\gamma$ runs over $\pi_1(C,p)$, and therefore by  the monodromy of $f$ and $g$  up to order $k+1$,
due to (\ref{eq:CompatibilityOrderk}). We deduce that 
invariants $u(x)$ and $\tilde u(x)$ of the two pairs coincide at $p$. On the other hand, 
$f$ and $g$ (resp. $\tilde f$ and $\tilde g$) coincide on a transversal.
Therefore, Lemma \ref{LemNormFormBifolBis} provides the existence and uniqueness of a diffeomorphism
$\Phi:(U,p)\to(\tilde U,p)$  conjugating $(f,g)$ to $(\tilde f,\tilde g)$ and inducing the identity on $C$. 
After analytic continuation, we can check that $\Phi$ is uniform, due to the fact that $(f,g)$ 
and $(\tilde f,\tilde g)$ have the same monodromy, and provides a global diffeomorphism $\Phi:U\to\tilde U$.
\end{proof}

\begin{proof}{\bf Realization part.}We proceed anagolously to the case $k=0$, taking into account the new invariants. Let us start from the model defined on $\C_x\times\C_y$ by
$$f=y\ \ \ \text{and}\ \ \ g= \frac{y}{\left(1-ku(x) y^k\right)^{1/k}}=y+u(x) y^{k+1}+\cdots$$
which are those first integrals satisfying $f(0,y),g(0,y)=y$ for the foliations
$$\F:\ dy=0\ \ \ \text{and}\ \ \ \G: dy+y^{k+1}\omega=0.$$
where $\omega=du$.
For all $\gamma$, we construct a diffeomorphism $\Phi_\gamma$ conjugating
$$(f^\gamma(x),g^\gamma(x))=(f(x+\gamma,0),g(x+\gamma,0))\ \ \ \text{with}\ \ \ 
(\rho_{\F}(\gamma)\circ f,\rho_{\G}(\gamma)\circ g).$$
By assumption, the two pairs have exactly the same invariants at any point $p$. 
Lemma \ref{LemNormFormBifolBis} provides the existence and unicity of $\Phi_\gamma$
at the neighborhood of any point $p\in\tilde C$, and therefore on a neighborhood of $\tilde C$.
\end{proof}

\begin{lemma} \label{L:conjugacy} When $\theta\not=0$, condition (\ref{eq:CompatibilityOrderk}) exactly means that 
$\rho_{\G}(\gamma)$ coincides with $\psi^{-1}\circ\rho_{\F}(\gamma)\circ\psi$
up to order $k+1$ where $\psi=y+cy^{k+1}+\cdots$.
\end{lemma}

\begin{proof}
Denoting $\varphi^\gamma_{\F}(y):=\rho_{\F}(\gamma)=a_\gamma y+\cdots$, a straightforward computation shows that
$$\psi^{-1}\circ\varphi^\gamma_{\F}\circ\psi\circ(\varphi^\gamma_{\F})^{-1}(y)=y+c(1-a_\gamma^{-k})y^{k+1}+\cdots\ \ \ 
\text{modulo}\ y^{k+2}$$
and condition (\ref{eq:CompatibilityOrderk}) reads
$$\varphi^\gamma_{\G}\circ(\varphi^\gamma_{\F})^{-1}(y)=y+c(1-a_\gamma^{-k})y^{k+1}+\cdots\ \ \ 
\text{modulo}\ y^{k+2}$$
for all $\gamma\in\pi_1(C,p)$, whence the conclusion.
\end{proof}

\begin{cor}
When $\theta\not=0$, then $(U,C)$ is analytically/formally linearizable.
\end{cor}

\begin{proof}
We deal with a pair $(\F,\G)$ of convergent foliation tangent at order $k+1$ along $C$, the proof of the formal analogue being the same. 

When $\theta\not=0$, the linear part $\mathrm{lin}\left(\rho_{\F}\right):\pi_1(C,p)\to\C^*$ 
is not unitary and has infinite image;  then both $\rho_{\F}$ and $\rho_{\G}$ are analytically
linearizable with the same linear part. One can then assume that $\rho_ {\F}$ coincide with its linear part in the pair $(\rho_ {\F},\rho_{G})$. Then, with the notation of lemma \ref{L:conjugacy},   
$H(\gamma):=\psi \circ\varphi_{\G}^\gamma\circ\psi^{-1}(y)=a_\gamma y\ \mod\ y^{k+2}$. One can then easily verify that there exists $\phi\in\Diff$ tangent to identity up to order $k+1$ such that for every $\gamma$,  $\phi\circ H(\gamma)\circ \phi^{-1}(y)=a_\gamma y$.
In particular, and thanks to Theorem \ref{ThmBifoliatedNeighborhood_k>0}, $\mathrm{lin}\left(\rho_{\F}\right)$ and $c$ provide a complete system of bifoliated analytical invariants. Let us show that we can realize the same invariants
for a pair of foliations $(\F_0,\G_0)$ on the linear neighborhood $(N_C,0)={\C}_x\times (\C_y ,0)/G$ where $G=\{g_\gamma, \gamma\in\Gamma_\tau\}$ is an abelian group of transformations of the form $g_\gamma (x,y)=(x+\gamma, \lambda_\gamma y)$ with ${\lambda_\gamma}^k=1$.  The coordinates $x$ and $y$ give rise on  $N_C$ to the transversal fibration $dx$ and the unitary foliation $dy$ (actually a fibration with $C$ as a  multiple fiber of order dividing $k$). 
Let $\lambda$ be a non zero complex number such that for every $\gamma\in\pi_1(C)$, $a_\gamma=\lambda_\gamma e^{\lambda \gamma}$.  Now,  consider on $N_C$  the foliations $\F_0$ and $\G_0$ respectively defined by 
$$\alpha=dy+{\lambda}y dx\ \text{and}\ \beta=\alpha-ck^2 y^{k+1}dx$$
whose tangency order along $C$ agrees with that of $\F$ and $\G$.
 
 Considering the new variable $z=ye^{{\lambda}{x}}$ ( a local first intergral for $\F_0$ which gives local model at $p=0$), 
one easily checks that $(U,C,\F,\G)$ and $\big((N_C,0), C, \F_0,\G_0\big)$ have the same invariants
and are thus equivalent. This gives the sought conjugation between $U$ and  $(N_C,0)$.
\end{proof}
\subsection{Proof of Theorem \ref{THM:PencilFol}}
In view of the results proved in section \ref{S:Abeliansubgroups} and \ref{S:Existence}, it remains to check that one obtains the right expression for the holonomies. Let us detail the proof for $p>1$ (the remaining cases being settled similarly). With the notations and results of Proposition \ref{prop:holonomyperiods}, one can choose a transversal $(T,z)$ in restriction to which $\omega_0=\frac{dz}{z^{k+1}}+\lambda\frac{dz}{z}$, so that $\varphi_\gamma(z)= \varphi_\gamma(z)=a_\gamma\cdot\exp\left(\sigma_{\omega_0}(\gamma) v_{k,\lambda} \right)$, 
with $\sigma_{\omega_0}(1)=0$ and $\sigma_{\omega_0}(\tau)=1$. 
Up to multiply $\omega_\infty$ by a constant, one can suppose that $\sigma_{\omega_\infty}(1)=1$ and  $\sigma_{\omega_\infty}(\tau)=\tau$. Then $\omega_t$ ($t\in\C$) in restriction to $(T,z)$ has the form $\omega_t =\frac{dz}{z^{k+1}}+\lambda\frac{dz}{z} +t\alpha \frac{dz}{z^{p+1}}dz +o(1)$, $\alpha\not= 0$.  Therefore, the $k+1$ jet of the  formal vector field $v_t$  dual to $\omega_t $ coincide with that of  $v_{k,\lambda}$. Using again item (2) of Proposition  \ref{prop:holonomyperiods}, one obtains that the $k+1$ jet of the holonomy of $\F_t$ along $\gamma\in K$ is given by the truncation at order $k+1$ of  $\exp{[(1+t)\left(
\sigma_{\omega_0(\gamma)} v_{k,\lambda} \right)]}$.   By the  comparison relation (\ref{eq:CompatibilityOrderk}) ($\theta=0$),  we indeed obtain  for every $\gamma\in\Gamma$,  that the $k+1$ jet of the holonomy of $\F_t$ along $\gamma$ coincide with the $k+1$ jet of $a_\gamma \exp{[(1+t)\left(
\sigma_{\omega_0 }(\gamma) v_{k,\lambda} \right)]}$. We thus have the sought relation for the expression of the holonomy of $\F_t  ,t\in\C$. We conclude along the same line in order to justify the computation of the holonomy of $\F_\infty$.\qed

\subsection{Proof of Theorem \ref{ThmBifoliated}} 

 Theorem \ref{ThmBifoliated} of the intoduction and more particularly its realization part is just a particular case of 
 Theorems \ref{ThmBifoliatedNeighborhood} and \ref{ThmBifoliatedNeighborhood_k>0}.\qed

\section{Formal classification of neighborhoods}
As recalled in the introduction, $(U,C)$ is formally linearizable as soon as $N_C$ has infinite order. We detail below the torsion cases.
\subsection{The case of fibrations}

Let $(U,C)$ be a neighborhood.
Assume that $C$ is a smooth (but possibly non reduced) fiber of an analytic fibration 
$f:U\to\Delta$ (with $\Delta\subset\C$ the unit disc): $f^{-1}(0)=mC$ and 
$N_C$ is torsion of order $m$, for some $m\in\Z_{>0}$.

\begin{prop}Let $(U,C)$ an analytic neighborhood with a fibration.
Are equivalent:
\begin{itemize}
\item the fibration $f:U\to\Delta$ is isotrivial,
\item there is a formal fibration $x:U\to C$ transversal to $C$,
\item $\Folf$ is not reduced to the fibration,
\item $(U,C)\simeq (N_C,0)$ is formally linearizable,
\item $(U,C)\simeq (N_C,0)$ is analytically linearizable.
\end{itemize}
\end{prop}

\begin{proof}If $(U,C)$ admits an analytic fibration transversal to $C$, then it is linearizable
(use $x:U\to C$ and $f:U\to\Delta$ as a pair of global coordinates). If $(U,C)$ admits a formal fibration 
transversal to $C$, then there exists one analytic by Grauert Comparison Theorem \cite{GrauertComparison}.
If $\F\in\Folf$ is distinct from the fibration, then it is transversal to it outside of $C$, inducing a fibration transversal 
to other fibers; the fibration is therefore isotrivial and the neighborhood linearizable.
On the linear neighborhood $(N_C,0)$, there are many foliations due to the size of automorphism group 
$$(x,f)\mapsto (x+\varphi(f),\psi(f))$$
but modulo this group, all foliations are diffeomorphic to 
$$df=0,\ \ \ \lambda\frac{df}{f} + dx=0\ \ \ \text{or}\ \ \ \frac{df}{f^{k+1}}+\lambda\frac{df}{f} + dx=0.$$
Indeed, given a foliation $\F\in\Folf$,
it is globally defined by 
$$df=\sum_{n\ge0}\alpha_nf^n$$
where $\alpha_n$ are holomorphic $1$-forms on $C$, therefore of the form $\alpha_n=a_ndx$.
Therefore, we can rewrite 
$$\frac{df}{\sum_n a_nf^n}-dx=0$$
(or $df=0$ if all $a_n=0$) and then normalize the meromorphic $1$-form by a change of $f$-coordinate.
In particular, we deduce that all representations are linearizable (in the logarithmic case) or normalizable (in the irregular case). 
\end{proof}

In the non isotrivial case, we have:

\begin{thm}\label{TH:F_n}
Let $(U,C)$ having a non isotrivial fibration $f:U\to\Delta$. Then $(U,C)$ is conjugated
 to one and only one of the following analytic neighborhood $F_n (a_1,a_\tau),\ n\geq 1$:
\[\left\{\begin{array}{ccc} (x,y)&\sim& (x+1,a_1 y) \\ (x,y)&\sim&(x+\tau+y^n,a_\tau y)\end{array}\right. \]
Moreover, the conjugation can be made analytic.
\end{thm}
\begin{proof}
Under the assumptions of Theorem \ref{TH:F_n}, Ueda has shown that $C$ is a (maybe multiple) fiber of some \textit{analytic} fibration (\cite[Theorem 3, p.596]{Ueda})  and the remainder follows from classical Kodaira-Spencer theory.  Remark that for $n>0$, there exists a transverse  fibration on the $n^{th}$ infinitesimal neighborhood $C_{(n)}$, but not on $C_{(n+1)}$, which is a way to prove that $F_n$ and $F_{n'}$ are formally distincts whenever $n\not=n'$. 
\end{proof}

\subsection{The case $N_C$ torsion and $\utype(C)$ finite}
The combination of Theorem \ref{thm:feuilletagesformels} and Proposition \ref{prop:holonomyperiods}
shows that holonomy representation of $\F_0$ takes the form in a convenient formal transversal $(T,y)$:
\begin{equation}\label{eq:HolonF0}
\rho_{\F_0}\ :\ 
\left\{\begin{matrix}
1\mapsto a_1y\hfill\\ \tau\mapsto a_\tau\exp(v_{k,\lambda})
\end{matrix}\right.\ \ \ \text{with}\ \ \ v_{k,\lambda}=\frac{y^{k+1}}{1+\lambda y^k}\partial_y
\end{equation}
where $(a_1,a_\tau)$ is the $m$-torsion monodromy of the unitary connection on $N_C$, 
$k=mk'$ is the Ueda type of $(U,C)$, $k'\in\Z_{>0}$, and $\lambda\in\C^*$. 

\subsubsection{The case $\F_\infty$ is a fibration transversal to $C$} \label{SS:transverse}
This is the classical suspension case and corresponds to the case $p=-1$ of theorem  \ref{THM:PencilFol}  : there are no other invariants and the neighborhood $(U,C)$
can be defined as the quotient of the germ $(\tilde U,\tilde C):=(\C_x\times \C_y,\{y=0\})$ 
by the group generated by
\[\left\{\begin{array}{ccc} (x,y)&\sim& (x+1,a_1 y) \\ (x,y)&\sim&(x+\tau,a_\tau \exp(v_{k,\lambda}))\end{array}\right. \]
and the pencil of foliations (or closed $1$-forms) is generated by
$$\omega_0=\frac{dy}{y^{k+1}}+\lambda\frac{dy}{y}\ \ \ \text{and}\ \ \ \omega_\infty=dx.$$

\subsubsection{The case $\F_\infty$ is a logarithmic foliation tangent to $C$}
This is the case $p=0$ in Theorem \ref{THM:PencilFol}:
\begin{equation}\label{eq:HolonFinfty_p=0}
\rho_{\F_\infty}\ :\ 
\left\{\begin{matrix}
1\mapsto a_1e^cy+o(y)\hfill\\ \tau\mapsto a_\tau e^{c\tau}y+o(y)
\end{matrix}\right.
\ \ \ \text{with}\ \ \ c\in\C^*.
\end{equation}
There are no other formal invariants in this case. In fact, let 
$$\varphi_{k,\lambda}:=\exp(v_{k,\lambda})\ \ \ \text{and}\ \ \ 
\displaystyle{g_{k,\lambda}(y):=\frac{1}{c}\int_0^y \frac{dt}{t}-a_\tau\varphi_{k,\lambda}^*\frac{dt}{t}}.$$
Then we have

\begin{thm}\label{thm:NormalFormNeigh0=p<k}
If $0=p<k$, then the neighborhood $(U,C)$ is formally equivalent to the quotient 
of $(\tilde U,\tilde C):=(\C_x\times \C_y,\{y=0\})$ by the group generated by
\begin{equation}\label{eq:NormalFormNeighGroup0=p<k}
\left\{\begin{array}{ccc} \phi_1(x,y)&=& (x+1\ ,\ a_1 y) \\ 
\phi_\tau(x,y)&=&(x+\tau+g_{k,\lambda}(y)\ ,\ a_\tau \varphi_{k,\lambda}(y))\end{array}\right. 
\end{equation}
and the pencil $\omega_t=\omega_0+t\omega_\infty$ of closed $1$-forms is generated by
\begin{equation}\label{eq:NormalFormNeigh1Form0=p<k}
\omega_0=\frac{dy}{y^{k+1}}+\lambda\frac{dy}{y}\ \ \ \text{and}\ \ \ \omega_\infty=cdx+\frac{dy}{y}
\end{equation}
with $c\in\C^*$. The holonomy representation of the pencil $\F_t:\{\omega_t=0\}$ is given by
\begin{equation}\label{eq:NormalFormNeighHol0=p<k}
\rho_{\F_t}\ :\ 
\left\{\begin{matrix}
1\mapsto a_1\exp(tv_{k,\lambda+t})\hfill\\ \tau\mapsto a_\tau\exp((t\tau+1) v_{k,\lambda+t})
\end{matrix}\right.\ \ \ \text{and}\ \ \ \rho_{\F_\infty}\ :\ 
\left\{\begin{matrix}
1\mapsto a_1e^c y\hfill\\ \tau\mapsto a_\tau e^{t\tau} y
\end{matrix}\right.
\end{equation}
when computed on the transversal $(T=\{x=0\},y)$.
\end{thm}

\begin{proof}The linear part of the holonomy (\ref{eq:HolonFinfty_p=0}) of $\F_\infty$ 
is not torsion; following Theorem \ref{TH:formalG}, $\rho_{\F_\infty}$ can be linearized
by formal change of coordinate. Using equivalence relation (\ref{eq:Bifoliatedk=0}),
this can be done without modifying the normal form (\ref{eq:HolonF0}) for the holonomy 
of $\F_0$. It follows that $\lambda$ and $c$ are the only formal invariants in the case 
$0=p<k$. The list of formal models (\ref{eq:NormalFormNeighGroup0=p<k}) for $(U,C)$
is obtained in a similar way that realization part of the proof of Theorem \ref{ThmBifoliatedNeighborhood}.
We first start with models (\ref{eq:NormalFormNeigh1Form0=p<k}) on $(\tilde U,\tilde C)$ 
and look for transformations of that neighborhood preserving $\tilde C$ and inducing the translation
lattice $\lattice$ on it, and commuting with the $1$-forms $\omega_0$ and $\omega_\infty$.
Commutation with $\omega_0$ with holonomy constraint (\ref{eq:HolonF0}) shows that 
$$\left\{\begin{array}{ccc} \phi_1(x,y)&=& (x+1+f(x,y)\ ,\ a_1 y) \\ 
\phi_\tau(x,y)&=&(x+\tau+g(x,y)\ ,\ a_\tau \varphi_{k,\lambda}(y))\end{array}\right. $$
with $f,g$ holomorphic, vanishing along $\tilde C:\{y=0\}$. Now, invariance of $\omega_\infty$
gives $f=0$ and $g=g_{k,\lambda}$. We check that the monodromy of first integrals
$f_0=y$ and $f_t=ye^{cx}$ for $\omega_0$ and $\omega_\infty$ on the quotient 
$(\tilde U,\tilde C)/<\phi_1,\phi_\tau>$ are given by $\rho_{\F_0}$ and $\rho_{\F_\infty}$,
so that, by the classification part of Theorem \ref{ThmBifoliatedNeighborhood},
we get formal equivalence:
$$(U,C,\F_0,\F_\infty)\sim(\tilde U,\tilde C,\F_{\omega_0},\F_{\omega_\infty})/<\phi_1,\phi_\tau>.$$
\end{proof}

\subsubsection{The remaining case $0<p<k$}
We now have
\begin{equation}\label{eq:HolonFinfty}
\F_{\infty}\ :\ 
\left\{\begin{matrix}
a_1[z+cz^{p+1}+o(z^{p+1})]\hfill\\ a_\tau[z+c\tau z^{p+1}+o(z^{p+1})]
\end{matrix}\right.
\end{equation}
with $c\in\C^*$ and $0<p<k$ with $p\in m\Z_{>0}$. Set $k=mk'$ and $p=mp'$. 
Given $P(z)=\sum_{i=0}^{p'}\lambda_i z^i$ a polynomial of degree $p'$ precisely, 
define
\begin{equation}\label{eq:PrincPartOneForm}
\omega_P:=P(\frac{1}{y^m})\frac{dy}{y},\ \ \ v_P:=\frac{y}{P(\frac{1}{y^m})}\partial_y%,\ \ \ \varphi_P:=\exp(v_P)
\ \ \ \text{and}\ \ \ g_{k,\lambda,P}(y):=\int_0^y a_\tau\varphi_{k,\lambda}^*\omega_P-\omega_P.
\end{equation}
The group $\Z_{k'}$ of $k'^{\text{th}}$ roots of unity acts on the set of polynomials $P$ as follows:
\begin{equation}\label{eq:RootUnityAction}
(\mu, P(z))\mapsto P(\mu^{-1} z)
\end{equation}

\begin{thm}\label{thm:NormalFormNeigh0<p<k}
If $0<p<k$, then there exist $\lambda\in\C$ and $P\in\C^{p'}\times\C^*$ unique up to the $\Z_{k'}$-action
(\ref{eq:RootUnityAction}) such that $(U,C)$ is formally equivalent to the quotient 
of $(\tilde U,\tilde C):=(\C_x\times \C_y,\{y=0\})$ by the group generated by
\begin{equation}\label{eq:NormalFormNeighGroup0<p<k}
\left\{\begin{array}{ccc} \phi_1(x,y)&=& (x+1\ ,\ a_1 y) \\ 
\phi_\tau(x,y)&=&(x+\tau+g_{k,\lambda,P}(y)\ ,\ a_\tau \varphi_{k,\lambda}(y))\end{array}\right. 
\end{equation}
The pencil $\omega_t=\omega_0+t\omega_\infty$ of closed $1$-forms is generated by
\begin{equation}\label{eq:NormalFormNeigh1Form0<p<k}
\omega_0=\frac{dy}{y^{k+1}}+\lambda\frac{dy}{y}\ \ \ \text{and}\ \ \ \omega_\infty=dx-\omega_P.
\end{equation}
The holonomy representation of the pair $(\F_0,\F_\infty)$ is given by
\begin{equation}\label{eq:NormalFormNeighHol0<p<k}
\rho_{\F_0}\ :\ 
\left\{\begin{matrix}
1\mapsto a_1 y\hfill\\ \tau\mapsto a_\tau \varphi_{k,\lambda}(y)
\end{matrix}\right.\ \ \ \text{and}\ \ \ \rho_{\F_\infty}\ :\ 
\left\{\begin{matrix}
1\mapsto a_1\exp(v_P)\hfill\\ \tau\mapsto a_\tau \exp(\tau v_P)
\end{matrix}\right.
\end{equation}
when computed on the transversal $(T=\{x=0\},y)$.
\end{thm}

\begin{proof}We first have to show 
that, under equivalence relation (\ref{eq:Bifoliatedk>0}), we can reduce the holonomy of $\F_\infty$
to the form (\ref{eq:NormalFormNeighHol0<p<k}). We can assume holonomy of $\F_0$ already normalized to (\ref{eq:HolonF0}), so that it remains to normalize the holonomy of $\F_\infty$ by help of some $\psi\in\Difffor$
with $\psi(y)=y+o(y^{k+1})$. Choose $\gamma\in\pi_1(C)$ such that $a_\gamma$ has order $m$
precisely, and consider $\varphi^\gamma:=\rho_{\F_\infty}(\gamma)$: 
we have $\varphi^\gamma(y)=a_\gamma y+c_0y^{p+1}+o(y^{p+1})$, for some $c_0\in\C^*$. If $\psi(y)=y+cy^{n+1}$, then 
$$\psi^{-1}\circ\varphi^\gamma\circ\psi(y)=\varphi^\gamma+a_\gamma c(1-a_\gamma^n)y^{n+1}+o(y^{n+1}).$$
By iterating such conjugacies with increasing $n$, this allow us to kill successively all coefficients 
of $\varphi^\gamma$ that are not of the form $y^{qm+1}$ (since $(1-a_\gamma^{qm})=0$).
Therefore, we can now assume $\varphi^\gamma(y)=a_\gamma y(1+y^p f(y^m))$; we note that $a_\gamma$
commutes with $\varphi^\gamma$, and therefore with $\varphi_0:=y(1+y^p f(y^m))$.
We can write $\varphi_0=\exp(v)$ for a unique formal vector field $v=y^{p+1}g(y^m)\partial_y$, with dual $1$-form
$\omega=\frac{h(y^m)}{y^{p+1}}dy$. Again, if $\psi(y)=y+cy^{n+1}$, we see that $\psi^*\omega=\omega-(ph(0)cy^{n-p-1}+o(y^{n-p-1}))dz$. By successive change of coordinates with $n=qm>p$, we can kill all positive coefficients
of $\omega$ so that it remains only the principal part $\omega_P=P(\frac{dy}{y^{m}})\frac{dy}{y}$ with $\deg(P)\le p$.
Finally, we have constructed some $\psi$ conjugating $\omega$ with $\Omega_P$, and therefore $v$ to $v_P$,
and $\varphi_0$ to $\exp(v_P)$. Since $\psi$ commutes with $a_\gamma y$, it is also conjugating $\varphi^\gamma$
to $a_\gamma\exp(v_P)$. By Theorem \ref{TH:formalG}, the whole holonomy representation, being in the centralizer
of $\varphi^\gamma$, must be normalized to $a\exp(t v_P)$ with $t\in\C$, $a^m=1$.
Here we use that $v_P$ can be normalized to some $v_{p,\mu}$ by a convenient $\psi(y)=y u(y^m)$.
We have thus reduced the pair of holonomy representations to the form (\ref{eq:NormalFormNeighHol0<p<k}).
One easily checks that the reduction is unique except that we could have used $\psi$ in the centralizer of $\rho_{\F_0}$,
i.e. $\psi(y)=ay+o(y^{p+1})$ with $a^m=1$. This gives rise to an action of $k^{\text{th}}$ roots of unity on the set
of principal parts $P$, which actually factors via an action of $k'^{\text{th}}$ roots of unity.
This ends the formal classification. It remains to prove holonomy representations of the pair of foliations 
$(\F_{\omega_0},\F_{\omega_\infty})$ defined by (\ref{eq:NormalFormNeigh1Form0<p<k}) 
on the quotient by (\ref{eq:NormalFormNeighGroup0<p<k}) is indeed given by (\ref{eq:NormalFormNeighHol0<p<k})
so that we can conclude by classification part of Theorem \ref{ThmBifoliatedNeighborhood_k>0}.
The holonomy of $\F_0$ computed on $x=1$ is clearly given by  (\ref{eq:NormalFormNeighHol0<p<k}).
For $\F_\infty$, the holonomy along the loop $1\in\lattice$ is also clearly like (\ref{eq:NormalFormNeighHol0<p<k}),
i.e. $\varphi^1=a_1\exp(v_P)$.
Again, by Theorem \ref{TH:formalG}, the whole holonomy representation $\rho_{\F_{\infty}}$, being in the centralizer
of $\varphi^1$, must be normalized to $a\exp(t v_P)$ with $t\in\C$, $a^m=1$. We conclude with (\ref{eq:HolonPencilUedakappa}) in Theorem \ref{THM:PencilFol} that holonomy along $\tau\in\lattice$ also fit with (\ref{eq:NormalFormNeighHol0<p<k}).
\end{proof}
\subsection{Proof of Theorem \ref{TH:FORMAL_CLASS}}

With the assumptions of theorem \ref{TH:FORMAL_CLASS}, let us fix the three formal basic invariants: the $m$-torsion monodromy representation attached to $N_C$, the Ueda type $k=mk'$ of $(U,C)$,  and the tangency order $p+1$ between the two distinguished foliations $\F_0$ and $\F_\infty$.

When $p=-1$, the calculation performed  in \ref{SS:transverse} shows that the set of formal equivalence classes is in one to one correspondance with the set of $k'+1$-uples of complex numbers of the form $(\lambda,0,.....,0)$ where $\lambda$ is the residue of $\omega_0$.

When $p=0$, Theorem \ref{thm:NormalFormNeigh0=p<k} provides a one to one correspondance between the set of formal equivalence classes and the set of $k'+1$-uples of the form $(\lambda,\lambda_0,0...0)$ with $\lambda_0=c\not=0$.
 
 The case $p=mp'>0$ is covered by Theorem \ref{thm:NormalFormNeigh0<p<k} where it is shown that the set of formal equivalence classes is in one to one correspondance with the set of $k'+1$-uples of the form $(\lambda,\lambda_0,...,\lambda_{p' },0,...,0)$ modulo the ${\Z}^{k'}$ action where the coefficient $\lambda_{p'}$ is not zero. As the possible values of $p'$ range over $\{1,....,k'-1\}$, this concludes the proof of the theorem.\qed

\section{About the analytic classification}\label{S:criteriaconv}

\subsection{Criteria of convergence}\label{sec:criteriaconv}

If we start with a convergent foliation $\F\in\Fol$, then one might ask when the formal $1$-form 
$\omega$ given by Corollary \ref{cor:folformalclosedform} is convergent. 
This is not true in general and one can find examples in \cite{Arnold} for instance.
In fact, the convergence of $\omega$ is equivalent to the existence of a convergent $\varphi\in\Diff$
conjugating the holonomy group $G$ to a subgroup of the models $\mb L$ or $\mb E_{k,\lambda}$
(see Theorem \ref{TH:formalG}). Here follow some criteria. The first one is a baby case of Bochner Theorem (see \cite[\S 2.1]{Frankpseudo}):

\begin{prop}\label{prop:Bochner}
If $G\subset\Diff$ is finite, then it is conjugated to a subgroup of $\mb L$.
\end{prop}

The next is a direct consequence of Koenigs Theorem (see \cite[\S 2.1]{Frankpseudo}):

\begin{thm}\label{thm:Koenigs}
If $G\subset\Diff$ is not unitary, i.e. contains an element $f(z)=az+\cdots$ with $\vert a\vert\not=1$,
then it is conjugated to a subgroup of $\mb L$.
\end{thm}

The last citerion is due to \'Ecalle and Liverpool (see \cite[\S 2.8.1]{Frankpseudo}):

\begin{thm}\label{thm:EcalleLiverpool}
If $G\subset\Diff$ is resonant, i.e. has only elements $f(z)=az+\cdots$ with $a^m=1$ for some $m$,
but is not virtually cyclic,
then it is conjugated to a subgroup of $\mb E_{k,\lambda}$.
\end{thm}

\subsection{Proof of Theorem \ref{thm:3convergent}}
First observe that, due to Theorem \ref{THM:PencilFol}, the holonomy of $\F_t$ is virtually cyclic
if, and only if $\frac{t}{1+t\tau}\in\Q$. On the other hand, if the holonomy $\rho_{\F_t}$ is not virtually cyclic, 
then Theorem \ref{thm:EcalleLiverpool} asserts that it preserves the (analytic) $1$-form
$\omega_{k,\lambda}$ and the closed $1$-form $\omega$ defining $\F_t$ constructed in 
Corollary \ref{cor:folformalclosedform} is convergent. If we have two such elements $t_1,t_2$ in the pencil,
then the full pencil can be recovered by the pencil of convergent closed meromorphic $1$-forms 
generated by $\omega_{t_1}$ and $\omega_{t_2}$. Moreover, the holonomies $\rho_{\F_{t_1}}$
and $\rho_{\F_{t_2}}$ being analytically conjugated to its formal normal form (\ref{eq:NormalFormNeighHol0=p<k})
or (\ref{eq:NormalFormNeighHol0<p<k}), we can use classification part of Theorem \ref{ThmBifoliatedNeighborhood}
or (\ref{ThmBifoliatedNeighborhood_k>0}) to conjugate $(U,C)$ analytically to its formal model of Theorem
(\ref{thm:NormalFormNeigh0=p<k}) or (\ref{thm:NormalFormNeigh0<p<k}).

Now, if $3$ elements of the pencil are convergent, say $\F_{t_1},\F_{t_2},\F_{t_3}$, we can also recover 
the full pencil analytically by considering those foliations $\F_t$ having constant cross-ratio with those $3$:
for each $c\in\C\setminus\{0,1,\infty\}$ and at each point $p\in U\setminus C$ close enough to $C$, 
there is a unique direction $T_p\F$ such that 
$$\mbox{cross-ratio}(T_p\F_{t_1},T_p\F_{t_2},T_p\F_{t_3},T_p\F)=c.$$
This defines a foliation $\F$ on the germ $U\setminus C$; let us show that it extends on the neighborhood of $C$.
We check it in the case where all $\Tang(\F_{t_i},\F_{t_j})=(k+1)[C]$
and leave as an exercise the case when one of the $3$ foliations is $\F_\infty$ with lower tangency $p+1$.
In local coordinates $(x,y)$ where $C$ is defined by $\{y=0\}$ and $\F_{t_i}$ by $\frac{dy}{dx}=s_i(x,y)$,
with say 
\begin{equation}\label{eq:LocForm3Fol}
s_1(x,y)=0,\ \ \ s_2(x,y)=y^{k+1}f(x,y)\ \ \ \text{and}\ \ \ s_3(x,y)=y^{k+1}g(x,y),
\end{equation}
$$\text{with}\ \ \ f,g,f-g\not=0, \  \ \ (\ \Leftrightarrow\ \Tang(\F_{t_i},\F_{t_j})=(k+1)[C])$$
the foliation $\F$ is defined by 
\begin{equation}\label{eq:InvCrossRat}
\frac{dy}{dx}=s(x,y)\ \ \ \text{with}\ \ \ c=\frac{\frac{s-s_1}{s_2-s_1}}{\frac{s-s_3}{s_2-s_3}}\ \Leftrightarrow\ s=\frac{s_1(s_2-s_3)-cs_3(s_2-s_1)}{(s_2-s_3)-c(s_2-s_1)}
.
\end{equation}
Clearly, if $s_1,s_2,s_3$ satisfy (\ref{eq:LocForm3Fol}), then 
$$\F\ :\ \frac{dy}{dx}=s=y^{k+1}\frac{cfg}{(c-1)f+g}$$
which is clearly holomorphic for all $c\not=1-g(0,0)/f(0,0)$. In particular, $\F$  belongs to $\Fol$. For obvious reasons of cardinality, one can find among those $c$ two foliations whose holonomy is not virtually cyclic and we can apply the previous reasoning.
\qed

\subsection{Examples of huge moduli} 

Recall (see \cite[Section 2]{Frankpseudo}) that the moduli space of germs 
\begin{equation}\label{eq:diffeoLinearParta}
\varphi(z)=az+\sum_{n>1}a_nz^n\in\Diff,\ \ \ \vert a\vert=1
\end{equation}
for analytic classification is infinite dimensional whenever $a$ is periodic, 
or does not satisfy the Brjuno diophantine condition. Let us be more precise.

If $a=e^{2i\pi\alpha}$, $\alpha\in\R\setminus\Q$, denote by $\left(\frac{p_n}{q_n}\right)_{n\in\Z_{\ge 0}}$, 
$p_n,q_n\in\Z$, $q_n>0$,
the sequence of truncations of the continued fraction for $\alpha$ (see \cite{PerezMarco}):
this is the fastest approximation sequence of $\alpha$ by rational numbers.
Then Brjuno condition writes
\begin{equation}\label{eq:Brjuno}
(\mathcal B)\ :\ \sum_{n\ge 0}\frac{\log(q_{n+1})}{q_n}<\infty
\end{equation}
and means that the approximation is not very fast. Brjuno proved that, for $a$ satisfying 
condition $(\mathcal B)$, any germ (\ref{eq:diffeoLinearParta}) is analytically linearizable.
On the other hand, Yoccoz proved the sharpness of this condition: for any $a$ non periodic
and not satisfying $(\mathcal B)$, then there exists non linearizable germs (\ref{eq:diffeoLinearParta}).
More precisely, Yoccoz constructs examples with infinitely many periodic points accumulating
on $0$, and can even prescribe, for each periodic orbit, the multiplicator of the return map
(see \cite{Yoccoz,PerezMarco}), giving an infinite number of parameters of freedom. 
In other words, as soon as Brjuno condition $(\mathcal B)$ is violated, the moduli space is infinite dimensional.

\begin{cor}For certain non torsion bundles $L\in\mathrm{Jac}(C)$ not satisfying diophantine condition (\ref{eq:diophantine}),
and for any $t_0\in\C^*$, there exists an infinite dimensional family of non analytically equivalent neighborhoods $(U,C)$
with $N_C=L$ and only $\F_0$ and $\F_{t_0}$ are convergent in the formal pencil $\Folf$. 
In particular, the transversal fibration $\F_\infty$ is divergent.
\end{cor}

The similar result with a convergent fibration, i.e. with $t_0=\infty$, is well-known (see \cite{Arnold}) 
and can be obtained by suspension method.

\begin{proof} Choose $N_C=L$ such that, in the linear model, the holonomy of the unitary foliation
$\F_0$ (\ref{eq:holPencilLinear}) is cyclic, generated by $y\mapsto ay$ with $a$ violating condition $(\mathcal B)$.
Then we can realize all Yoccoz non linearizable dynamics as holonomy of $\F_0$ simultaneously with the linear holonomy
for $\F_{t_0}$ by using realization part of Theorem \ref{ThmBifoliatedNeighborhood}. If another formal foliations
$\F_{t_1}$ were convergent, then its holonomy representation (being non unitary) would be linearizable by Theorem \ref{thm:Koenigs}
and we could conclude that $(U,C)$ was itself analytically linearizable by applying classification part 
of Theorem \ref{ThmBifoliatedNeighborhood} to the pair $(\F_{t_0},\F_{t_1})$; this would contradict 
that holonomy of $\F_0$ is not linearizable. In a similar way, if the formal transversal fibration $\F_\infty$
were convergent, then $(U,C)$ would be the suspension of the linear holonomy of $\F_{t_0}$, and 
therefore itself linearizable; contradiction again.
\end{proof}

Diophantine condition (\ref{eq:diophantine}) for the bundle $N_C$ implies Siegel diophantine condition 
(see \cite{PerezMarco}) for the holonomy of its unitary foliation $\F_0$. We do not know what is the optimal 
diophantine condition for $N_C$ to be linearizable. There are two problems. First, we do not know what 
is the optimal condition on a pair $(a_1,a_\tau)\in\C^*\times\C^*$ such that there is a non linearizable
abelian representation with these multiplicators (see \cite{Moser,PerezMarco}); when the group $<a_1,a_\tau>$ 
is not cyclic, we might 
expect more constraint than just violating condition $(\mathcal B)$ for both $a_1$ and $a_\tau$.
Secondly, we do not know if there exists examples of neighborhoods $(U,C)$ with divergent $\F_0$,
so that even assuming that $a_1$ or $a_\tau$ satisfies  condition $(\mathcal B)$, we cannot conclude
that all neighborhoods are linearizable.

Let us now turn to the torsion case $a^m=1$. For each prescribed formal model $a\exp(v_{k,\lambda})$,
there is an explicit moduli space of infinite dimension for the analytic conjugacy in $\Diff$. This has been 
described in various ways by Voronin, \'Ecalle, Martinet, Ramis, Malgrange \cite{Voronin,MartinetRamis,Ecalle,Malgrange}
(see also \cite[Section 2.7]{Frankpseudo}): the moduli space is more or less identified with $\Diff^{2k}$.
We have a similar description of the moduli space for representations $\pi_1(C)\to\mb E_{k,\lambda}$
as soon as the image is virtually cyclic, therefore completing the statement of Theorem \ref{thm:EcalleLiverpool}
(see \cite[Sections 2.8 and 2.10]{Frankpseudo}).

\subsection{Proof of Theorem \ref{thm:EmbedEcalleVoronin}}We assume $t_1,t_2\in\C$ and let the case
$t_1$ or $t_2=\infty$ for the reader. Since the holonomy of $\F_{t_1}$ is assumed to 
be virtually cyclic, there is a huge moduli of non analytically conjugated representations in $\Diff$. Using 
realization part of Theorem \ref{ThmBifoliatedNeighborhood_k>0}, we can realize such holonomy representations
for $\F_{t_1}$ simultaneously with the formal model for $\F_{t_2}$. After convenient conjugacy by a polynomial 
diffeomorphism $\phi\in\Diff$ of degree $k+1$, we can assume that the pair of holonomies coincide up to order $k$ 
and fit with the formal invariants $P$ of Theorem \ref{thm:NormalFormNeigh0<p<k}.\qed

\begin{remark}In case $\frac{t_i}{1+t_i\tau}\in\Q$, $i=1,2$, we can deform independently the two holonomy
representations following the method presented in the proof of Theorem \ref{thm:EmbedEcalleVoronin}. Therefore, we can explicitely describe
the moduli space of neighborhoods with $2$ convergent foliations when $N_C$ is torsion. However, in the 
non torsion case, even the moduli space of diffeomorphism germs of type (\ref{eq:diffeoLinearParta}) with 
$a$ not satifying $(\mathcal B)$ is totally unknown: even the regularity properties of its structure are widely open.
We just know that we can inject Ecalle-Voronin moduli: it suffices, for each resonant germ $\psi$, 
to use Yoccoz method to construct a germ $\varphi(z)=az+o(z)$ with infinitely many periodic orbits 
accumulating on $z=0$ with all return maps analytically equivalent to $\psi$.
\end{remark}

\section*{Appendix: Action of the involution on formal models}
For each prescribed normal bundle $N_C$ and  $k=\utype(U,C)<\infty$, we have given a list 
$(U_{\lambda,\Lambda},C)$ of models for the classification up to formal equivalence
inducing the identity on $C$ (see (\ref{eq:equivNeighbId})).
If we relax the fact that $\phi$ induces the identity on $C$,
we have to take into account the action of $\mbox{Aut}(C)$. First of all, the two closed meromorphic $1$-forms 
$\omega_0$ and $\omega_\infty$ given in (\ref{eq:NormalFormNeigh1Form0<p<k}) induce by duality two vector fields 
$$v_0=\partial_x\ \ \ \text{and}\ \ \ v_\infty=\frac{1}{\omega_0(v_\Lambda)}(\partial_x+v_\Lambda)
$$
which generate a $2$-dimensional group of symetries for $(U_{\lambda,\Lambda},C)$
acting transitively on $C$ and its complement $U\setminus C$; 
it induces an action by translation on $C$, but acting trivially on $H^1(C)$.
On the other hand, $\mbox{Aut}(C)$ acts on $H^1(C)$ as a cyclic group of order $q=2, 3, 4$ or $6$.
We would like to investigate the induced action on the set of normal forms.
Except the $3$ exceptional configurations associated to $q=3,4,6$, recall that $q=2$ generically.  
For convenience, denote by $U(a_1,a_\tau,\lambda,\Lambda)$ the neighborhood $(U_{\lambda,\Lambda},C)$ where $(a_1,a_\tau)$ generate the unitary monodromy on $N_C$ (and therefore determine $N_C$), assuming that the group $<a_1,a_\tau>$ has order $m=\ord(N_C)$, and $k=\utype(U,C)$ is the Ueda type. Note that we have a well defined action of ${\mathbb C}^*$ on $\Lambda$ defined, for $\xi\in{\C}^*$ by $\xi . (\lambda_0,\lambda_1...,\lambda_{k'-1})=(\lambda_0,{\xi}^{-1}{\lambda}_1,...,\xi^{-(k'-1)}{\lambda}_{k'-1})$

Then one checks directly that the action of the involution induces a diffeomorphism
$$\phi:U(a_1,a_\tau,\lambda,\Lambda)\to U(a_1^{-1},a_\tau^{-1},-\lambda,-(\xi^m .\Lambda))\ $$
defined by $\ (x,y)\mapsto(-x,\xi y)$, $\xi^k=-1$ (the Ueda type  remains unchanged by this conjugacy map). \begin{remark}
In the exceptional cases $q>2$ and when $\utype (C)<\infty$, it is quite uneasy to explicit conjugacy map associated to  transformation of the elliptic curve of finite order $>2$. 
However, it is still possible to describe equivalent normal forms by the bifoliated method, 
but this involves rather complicated formulae. 
As an example,  the reader may amuse himself to prove that, in Theorem \ref{thm:NormalFormNeigh0<p<k} 
and when $N_C=\mathcal O_C$,
and for $\tau=i$ :
\begin{enumerate}
\item When $p=1$, $U(\lambda, \Lambda)\sim U(-i\lambda -\lambda_0, i(\lambda_0,{\xi}^{-1}\lambda_1),0,...,0),\ \xi=\exp{3i\pi/4}$.
\item When $k=3$, $p=2$, $U(\lambda, \Lambda)\sim U(-i\lambda -\lambda_0, i(\lambda_0,-i\lambda_1+\frac{{\lambda_2}^2}{2},-\lambda_2))$
\end{enumerate}
\end{remark}

%\bibliography{Bifoliated}
%\bibliographystyle{plain}

\end{document}